\documentclass[12pt]{amsart}
\usepackage{amsthm,amsmath,amssymb,amscd,graphics,enumerate, stmaryrd,xspace,verbatim, epic, eepic,url,fullpage}

\usepackage{enumitem}

\usepackage[usenames,dvipsnames]{color}

\usepackage[colorlinks=true]{hyperref}
\usepackage{cleveref}

\usepackage[active]{srcltx}
\usepackage[all]{xypic}
\SelectTips{cm}{}
\usepackage{quiver}

{
   \newtheorem{theorem}{Theorem}[section]
      \newtheorem*{theorem*}{Theorem}
   \newtheorem{proposition}[theorem]{Proposition}

   \newtheorem{lemma}[theorem]{Lemma}

   \newtheorem{corollary}[theorem]{Corollary}
   
   \newtheorem*{conjecture*}{Conjecture}

}
{\theoremstyle{definition}
          \newtheorem*{exercise*}{Exercise}
   
   \newtheorem{example}[theorem]{Example}
 
   \newtheorem*{example*}{Example}
   
   \newtheorem{definition}[theorem]{Definition}

   \newtheorem*{definition*}{Definition}
   
   \newtheorem{remark}[theorem]{Remark}

}
\newcommand{\CC}{{\mathbb{C}}}
\newcommand{\QQ}{{\mathbb{Q}}}

\newcommand{\PP}{{\mathbb{P}}}
\newcommand{\ZZ}{{\mathbb{Z}}}

\renewcommand{\AA}{{\mathbb{A}}}

\def\Bl{{\rm Bl}}

\newcommand{\cO}{{\mathcal O}}

\def\<{\langle}
\def\>{\rangle}

\newcommand{\Spec}{\operatorname{Spec}}

\newcommand{\Proj}{\operatorname{Proj}}

\newcommand{\Aut}{{\operatorname{Aut}}}

\newcommand{\codim}{\operatorname{codim}}

\newcommand{\red}{{\operatorname{red}}}

\def\:{{\colon}}
\def\.{{,\dots,}}
\def\dim{{\rm dim}}

\def\Ex{\mathrm{Ex}}

\newcommand{\double}{\genfrac..{0pt}1
{\raise -1pt\hbox{$\scriptstyle\longrightarrow$}}{\raise 3pt\hbox
{$\scriptstyle\longrightarrow$}}}

\newcommand{\setmin}{\smallsetminus}

\renewcommand{\setminus}{\smallsetminus}

\renewcommand\H{\operatorname{H}}

\def\int{{\rm int}}

\def\tototi{\mathbin{\mathop{\otimes}\limits^{\raise-1pt\hbox
{$\scriptscriptstyle {\rm L}$}}}}

\def\indlim{\mathop{\vrule width0pt height7pt depth
4pt\smash{\lim\limits_{\raise 1pt\hbox to 14.5pt
{\rightarrowfill}}}}}
\def\projlim{\mathop{\vrule width0pt height7pt depth
4pt\smash{\lim\limits_{\raise 1pt\hbox to 14.5pt
{\leftarrowfill}}}}}

\newcommand\displaceamount{3pt}

\newcommand{\doubledown}{\ar@<\displaceamount>[d]\ar@<-\displaceamount>[d]}

\newcommand{\doubleup}{\ar@<\displaceamount>[u]\ar@<-\displaceamount>[u]}

\newcommand{\doubleright}{\ar@<\displaceamount>[r]\ar@<-\displaceamount>[r]}

\newcommand{\res}{{\operatorname{res}}}

\def\Xbar{\overline{X}}

\def\Supp{\mathrm{Supp\ \!}}
\def\Diff{\mathrm{Diff}}
\def\Bir{\mathrm{Bir}}

\def\phibar{\overline{\phi}}
\def\psibar{\overline{\psi}}
\def\Ybar{\overline{Y}}
\def\Zbar{\overline{Z}}

\def\Bs{\mathrm{Bs}}

\definecolor{AL}{RGB}{10,200,100}

\numberwithin{equation}{section}

\begin{document}

\title{Finite generation of relative log canonical algebras of semi-dlt pairs}

\author{Alberto Landi}
\address{Department of Mathematics, Box 1917, Brown University,
Providence, RI, 02912, U.S.A}
\email{alberto\_landi@brown.edu}


\maketitle

\begin{abstract}
    We prove that the relative log canonical algebra of a semi-dlt pair $(X,\Delta_X)$ projective over a scheme $T$ is finitely generated, equivalently, that $(X,\Delta_X)$ admits a relative stable model, provided that the normalization admits a log canonical model over $T$, that $(X,\Delta_X)$ admits a stable model over an open dense subscheme $T^0\subset T$ whose exceptional locus does not contain any stratum of the conductor, and that the log centers contained in the conductor have images meeting $T^0$. Our main contribution is a direct proof that avoids Koll\'ar's gluing theory entirely.

    As a consequence, we recover demi-normal versions of results of Hacon--Xu and Birkar.
    Furthermore, we derive the existence of certain MMP steps for slc pairs, recovering results of Ambro and Koll\'ar. Finally, these results lay the groundwork for a streamlined proof of the properness of the Koll\'ar--Shepherd-Barron--Alexeev moduli space of stable pairs, to be completed in forthcoming work.
\end{abstract}

\tableofcontents

\section{Introduction}

Given a projective \emph{log canonical (lc) pair} $(X,\Delta_X)$ and a projective morphism $f\colon X\rightarrow T$, finite generation of the relative \emph{log canonical algebra} $R_X:=R(X/T,K_X+\Delta_X)$ is a fundamental problem in higher-dimensional algebraic geometry. The minimal model program (MMP) provides powerful methods for establishing finite generation in many important cases, in particular for smooth varieties~\cite{BCHM,HM2,Cascini-Lazic_NewOutlookI}.
When $R_X$ is finitely generated and $K_X+\Delta_X$ is $f$-big, the induced birational map $X\dashrightarrow\Proj R_X$ yields a \emph{log canonical model over $T$}. The target carries a natural structure of an lc pair for which the log canonical divisor is ample over $T$; when $T$ is a point, such a pair is called \emph{stable}. Stable pairs form the basic objects parametrized by the \emph{Koll\'ar--Shepherd-Barron--Alexeev (KSBA) moduli space}~\cite{KSB88,Alexeev_MgnW_surfaces,Kollar-moduli}, which plays a central role in the moduli theory of higher-dimensional varieties.

As is well known in the case of curves, obtaining a compact moduli space requires allowing degenerations to reducible or self-intersecting varieties. This leads naturally to the notion of \emph{semi-log-canonical} (\emph{slc}) pairs, introduced in~\cite{KSB88} for surfaces. The underlying variety of an slc pair is \emph{demi-normal}, that is, it satisfies Serre's $S_2$ condition and has at worst ordinary double points in codimension 1.
Given an slc pair $(X,\Delta_X)$, its normalization pair $(\Xbar,D_{\Xbar}+\Delta_{\Xbar})$ is log canonical, where $D_{\Xbar}$ denotes the \emph{conductor} divisor. The slc pair $(X,\Delta_X)$ is then recovered by gluing the normalization pair along a natural involution $\mu$ on $(D_{\Xbar},\Diff_{D_{\Xbar}}(\Delta_{\Xbar}))$. The KSBA moduli space compactifies the moduli of stable lc pairs by parametrizing stable slc pairs.

A systematic treatment of slc pairs is substantially more difficult, as MMP steps cannot always be executed~\cite[Examples 4, 5]{Ambro-Kollar-slc-MMP} and finite generation of the log canonical ring can fail even for normal crossing surfaces without boundary~\cite{Kollar_examples_SNC}. Nevertheless, substantial progress has been made in the study of slc pairs since their introduction~\cite{FujinoAbundanceSLC,Kollar-singularities,Gongyo_abundance,fujino_theorems_slc,FG_fin_Brepr,Fujita_semi-terminal_modifications,Hacon_Xu_finBrepr_slc_abundance,Ambro-Kollar-slc-MMP,Hashizume_sdlt_models}.

\emph{Koll\'ar's gluing theory} provides a powerful framework for overcoming these problems. In particular, it gives sufficient conditions under which a triple $(\Xbar,D_{\Xbar}+\Delta_{\Xbar},\mu)$ arises as the normalization data of an slc pair; we refer to~\cite[\S4.5, \S4.6, \S5, \S9]{Kollar-singularities} for an extensive treatment. The gluing machinery has played an important role in many results about the birational geometry of slc pairs.
In this paper, we show that, in several important cases, the same finite-generation problem can be treated directly, bypassing this machinery entirely.

With a view towards an application to the KSBA moduli space of stable pairs, we establish a sufficient condition for finite generation of the log canonical algebra of semi-dlt pairs (Definition~\ref{def: dlt/sdlt/sklt}) over a base $T$. We refer to Definition~\ref{def: stable model slc pairs} and Lemma~\ref{lem: redundant def stable model} for the notion of a \emph{good stable model}, which serves as a demi-normal analogue of the log canonical model and is characterized as the Proj of the log canonical algebra in Proposition~\ref{prop: Y=proj RX good stable model} and Theorem~\ref{thm: projRX is stable model}.

\begin{theorem}\label{thm: general version}
    Let $(X,\Delta_X)$ be a semi-dlt pair, $T$ a semi-normal scheme that is separated and essentially of finite type over $\CC$, and $f\colon X\rightarrow T$ a projective morphism. Let $T^0\subset T$ be an open dense subscheme, and set $X^0=T^0\times_{T}X$, $\Delta_{X^0}=\Delta_X|_{X^0}$. Assume that $X^0$ intersects
    all log centers of $(X,\Delta_X)$ lying on the conductor $D_X$; in particular, $X^0$ intersects all strata of $D_X$.
    Suppose that $(X^0,\Delta_{X^0})$ admits a good stable model
    \[
    \begin{tikzcd}
        \phi^{0,c}\colon (X^0,\Delta_{X^0})\arrow[r,dashed] & (X^{0,c},\Delta_{X^{0,c}})
    \end{tikzcd}
    \]
    whose exceptional locus does not contain any stratum of $D_{X^0}$, and that the normalization pair $(\Xbar,D_{\Xbar}+\Delta_{\Xbar})$ admits a log canonical model over $T$.
    
    Then, $R_X:=R(X/T,K_X+\Delta_X)$ is finitely generated, and $\Proj R_X$ is the good stable model of $(X,\Delta_X)$ over $T$.
\end{theorem}

In the case $\lfloor\Delta_X\rfloor=0$, Theorem~\ref{thm: general version} admits an equivariant strengthening with respect to finite groups of B-birational self-maps of the normalization pair.
Recall that if $(Y,\Delta_Y)$ is an lc pair and $f\colon Y\rightarrow T$ a proper morphism, we denote by $\Bir(Y/T,\Delta_Y)$ the group of birational automorphisms of $Y$ over $T$ that preserve the boundary $\Delta_Y$; see Definition~\ref{def: B-birational maps}. There is an induced action of $\Bir(Y/T,\Delta_Y)$ on (a Veronese subalgebra of) the log canonical $\cO_T$-algebra of $(Y,\Delta_Y)$; see Lemma~\ref{lem: Bir acts on log canonical} and Definition~\ref{def: invariants under Bir}.

\begin{theorem}\label{thm: equivariant sklt}
    With the notation of Theorem~\ref{thm: general version}, assume further that $\lfloor\Delta_X\rfloor=0$.
    Then, for every finite subgroup $G\subset\Bir(\Xbar/T,D_{\Xbar}+\Delta_{\Xbar})$, the inclusion
    \[
    \begin{tikzcd}
        R(X/T,K_X+\Delta_X)\cap R(\Xbar/T,K_{\Xbar}+D_{\Xbar}+\Delta_{\Xbar})^G\arrow[r,hookrightarrow] & R(\Xbar/T,K_{\Xbar}+D_{\Xbar}+\Delta_{\Xbar})
    \end{tikzcd}
    \]
    is a finite extension. In particular, $R_X:=R(X/T,K_X+\Delta_X)$ is finitely generated, and $\Proj R_X$ is the good stable model of $(X,\Delta_X)$ over $T$.
\end{theorem}

\subsection{Applications}\label{subsec: intro applications}

We discuss several consequences of Theorem~\ref{thm: general version}. In particular, we recover demi-normal version of results of Hacon--Xu~\cite[Theorem 1.1, Theorem 1.6]{Hacon_Xu_dslt} and Birkar~\cite[Theorem 1.1(3)]{Birkar_existenceflips}, with applications to the properness of the KSBA moduli space and to the existence of MMP steps for slc pairs, respectively.

\subsubsection{Properness of the KSBA moduli space of stable pairs}

It is known that the KSBA moduli space is proper.
The proof of properness for irreducible components whose general member is normal uses resolution of singularities to establish a \emph{semi-stable reduction} result~\cite{Abramovich-deJong,Karu-thesis,AK}, followed by an application of~\cite[Theorem 1.1]{Hacon_Xu_dslt} to obtain a \emph{stable family}; see~\cite[\S2]{Kollar-moduli} and~\cite{Karu-boundedness-stable-varieties}. When the general member is merely demi-normal, the existing proof passes first to the normalization, applies the recipe above, and finally glues the resulting family together using Koll\'ar's gluing machinery; see for instance~\cite[Theorem 2.51]{Kollar-moduli} or~\cite[\S7]{Hacon_Xu_dslt}.

Applying this strategy directly at the demi-normal level presents two difficulties. First, semi-stable reduction is not available in the non-normal setting.
This obstacle will be addressed in forthcoming joint work with Roberta Pagliaro, Simon Stojkovic, and Hao (Nick) Sun, leveraging recent resolution techniques involving DM-stacks~\cite{BDS-B-nc,Wlodarczyk-nc,AT-umbrellas}.

Second, passing from a semi-stable family to a stable family requires finite generation of the log canonical algebra of the total space, for which no direct proof avoiding Koll\'ar's gluing theory was previously available.
The following demi-normal analogue, for log big pairs, of~\cite[Theorem 1.1]{Hacon_Xu_dslt} provides such a proof and follows immediately from Theorem~\ref{thm: general version}.

\begin{corollary}\label{cor: main intro}
    Let $(X,\Delta_X)$ be a semi-dlt pair, $T$ a semi-normal scheme that is separated and essentially of finite type over $\CC$, and $f\colon X\rightarrow T$ a projective morphism. Let $T^0\subset T$ be an open dense subscheme, and set $X^0=T^0\times_{T}X$, $\Delta_{X^0}=\Delta_X|_{X^0}$. Assume that $X^0$ intersects all log canonical centers of $(X,\Delta_X)$, and all log centers lying on $D_X$; in particular, $X^0$ intersects all strata of the conductor $D_X$.
    Suppose that $(X^0,\Delta_{X^0})$ admits a good stable model
    \[
    \begin{tikzcd}
        \phi^{0,c}\colon (X^0,\Delta_{X^0})\arrow[r,dashed] & (X^{0,c},\Delta_{X^{0,c}})
    \end{tikzcd}
    \]
    over $T^0$ whose exceptional locus does not contain any stratum of $D_{X_0}$.
    
    Then, $R_X:=R(X/T,K_X+\Delta_X)$ is finitely generated, and $\Proj R_X$ is the good stable model of $(X,\Delta_X)$ over $T$.
\end{corollary}

For the intended application to the KSBA moduli space, we will use the following special case of Corollary~\ref{cor: main intro}.
When the base is a regular curve $T$, an slc pair $(X,\Delta_X)$ with a flat proper morphism $X\rightarrow T$ is an \emph{slc family} if $(X,\Delta_X+X_t)$ is an slc pair for all $t\in T$, see~\cite[Definition/Theorem 2.3]{Kollar-moduli}. We refer to \cite{Kollar-moduli} for the general notion of slc family.

\begin{corollary}\label{cor: simple normal crossing case}
    Let $(Y,\Delta_Y)$ be a demi-normal pair, flat and projective over a regular 1-dimensional scheme $T$ that is separated and essentially of finite type over $\CC$, and suppose that every irreducible component of $Y$ is normal in codimension 1. Suppose that there is a dense open subset $T^0\subset T$ such that $(Y^0,\Delta_{Y^0})=(Y,\Delta_Y)|_{T^0}$ is a stable slc family over $T^0$.
    
    Let $\phi\colon X\rightarrow Y$ be a birational morphism over $T$ that does not contract any stratum of the conductor $D_X$. Set
    \[
        \Delta_X=\phi_*^{-1}\Delta_Y+\sum a_iE_i,
    \]
    where $E_i$ are the irreducible $\phi$-exceptional divisors dominating $T$ and satisfying
    \[
        a_i=a(E|_{Y^0},Y^0,\Delta_{Y^0})>0.
    \]
    Suppose further that for every closed point $t\in T$ the pair $(X,\Delta_X+X_t)$ is slc.
    
    Then, $(X,\Delta_X)$ admits a stable model $(X^c,\Delta_{X^c})$ over $T$ that is also a stable family, and it is isomorphic to $(Y^0,\Delta_{Y^0})$ over $T^0$.
\end{corollary}

\subsubsection{MMP for slc pairs}

Another application of Theorem~\ref{thm: general version} is that it allows us to run a special version of MMP for slc pairs, in a manner similar to~\cite{Ambro-Kollar-slc-MMP}, but without relying on Koll\'ar's gluing machinery. Since the main results are already present in loc.\! cit.\!, we avoid restating them here, and refer to \S\ref{sec:slc-MMP}.

In this subsection, we limit ourselves to stating the demi-normal version of a special case of the result~\cite[Theorem 1.1(3)]{Birkar_existenceflips} of Birkar and~\cite[Theorem 1.6]{Hacon_Xu_dslt} of Hacon and Xu.

\begin{corollary}\label{cor: analog Birkar}
    Let $f\colon X\rightarrow T$ be a projective generically-finite morphism between projective demi-normal schemes over $\CC$. Let $\Delta_X$ and $\Delta_X'$ be effective $\QQ$-divisors on $X$ such that $(X,\Delta_X+\Delta_X')$ is slc and $(X,\Delta_X)$ is semi-dlt. Suppose that $K_X+\Delta_X+\Delta_X'\sim_{\QQ,T}0$, and that $f$ is finite at the generic point of every log center lying on $D_X$. Then, $R_X:=R(X/T,K_X+\Delta_X)$ is an $\cO_T$-algebra of finite type, and $X\dashrightarrow\Proj R_X$ is the stable model of $(X,\Delta_X)$ over $T$.
\end{corollary}

\subsubsection{Complex Analytic Spaces}

While Koll\'ar's gluing theory has played a fundamental role in the algebraic setting, its extension to complex analytic spaces is not currently available in the same generality. In contrast, the main technical results used in the proof of Theorem~\ref{thm: general version} have direct analogues for complex analytic spaces~\cite{Fujino_MMPanalytic,Fujino_ConeContractionAnalytic,Fujino_RelativeLogPluricanonical,Fujino_QuasiLogAnalytic,Fujino_VanishingAnalytic,Enokizono-Hashizume-analytic-MMP,Enokizono_Hashizume_MMP_stacks}, making our approach naturally suited to this context. The author intends to extend the results of this paper to projective morphisms of complex analytic spaces in forthcoming work.

\subsection{Avoiding Koll\'ar's Gluing}\label{subsec: motivation}

We briefly explain how the main results from the literature used in the proof can be invoked without using Koll\'ar's gluing theory in our setting. The two main technical inputs are~\cite[Theorem 2]{Hacon_Xu_finBrepr_slc_abundance} and~\cite[Theorem 1.1]{Hacon_Xu_dslt}.

In~\cite[Theorem 2]{Hacon_Xu_finBrepr_slc_abundance}, Hacon and Xu prove that the log canonical divisor of an slc pair is semi-ample if and only if the same holds for its normalization pair. The authors use Koll\'ar's gluing to descend the morphism induced by the log canonical divisor on the normalization to the slc pair. In our arguments, however, this result is only needed when the base is projective. In this case, the required semi-ampleness statement was also proved by Fujino and Gongyo using different methods~\cite[Theorem 1.4]{FG_fin_Brepr},\cite{FujinoAbundanceSLC}; see also~\cite[Remark 1.5]{Hacon_Xu_finBrepr_slc_abundance}.

In~\cite[Theorem 1.1]{Hacon_Xu_dslt}, Hacon and Xu prove, in particular, that a dlt pair over a normal base $T$ admitting a log canonical model over a dense open $T^0$, admits a log canonical model over $T$, provided that no log canonical center is mapped into $T\setminus T^0$. The theorem relies on Koll\'ar's gluing only to establish the semi-ampleness result in~\cite[Proposition 3.1]{Hacon_Xu_dslt}, which is a special case of~\cite[Theorem 2]{Hacon_Xu_finBrepr_slc_abundance}. Thus, when the base is projective, we can replace \cite[Proposition 3.1]{Hacon_Xu_dslt} with~\cite[Theorem 1.4]{FG_fin_Brepr}, as explained above, and consequently use~\cite[Theorem 1.1]{Hacon_Xu_dslt} without invoking Koll\'ar's gluing theory.

For these reason, in Lemma~\ref{lem: reduction to projective case} we reduce Theorem~\ref{thm: general version} to the case where $T$ is a projective scheme over $\CC$.

We could have avoided this reduction by using the generalization~\cite[Theorem 1.1]{Fujino_MMPanalytic} of~\cite[Theorem 2]{Hacon_Xu_finBrepr_slc_abundance} to the analytic setting. Fujino's proof does not rely on Koll\'ar's gluing theory, as it is not even clear how well that machinery works for analytic spaces.

\subsection{Strategy of the proof}\label{subsec: proof strategy}

First, in Lemma~\ref{lem: reduction to projective case} we will reduce Theorem~\ref{thm: general version} to the case where $T$ is projective over $\CC$, which we will assume throughout.

We then deduce Theorem~\ref{thm: general version} from the semi-klt case Theorem~\ref{thm: equivariant sklt}. We fix an $f$-ample Cartier divisor $A_X$ and, for $0<\epsilon\ll1$, perturb the pair $(X,\Delta_X)$ to $(X,\Delta_X+\epsilon A_X)$, which is semi-klt by Lemma~\ref{lem: improving boundary by perturbation}. We apply Theorem~\ref{thm: equivariant sklt} to obtain a stable model $\phi_+^c\colon X\rightarrow X^c_+$ of the perturbed pair. For $0<\epsilon\ll1$, this model is independent of $\epsilon$. Setting $\Delta_{X^c_+}:=(\phi_+^c)_*\Delta_X$, we show that the divisor $K_{X^c_+}+\Delta_{X^c_+}$ on $X_+^c$ is $\QQ$-Cartier and semi-ample over $T$, and hence that its associated $\cO_T$-algebra is finitely generated. The main technical input is~\cite[Theorem 1.4]{FG_fin_Brepr}: over a projective base, the log canonical divisor of an slc pair is semi-ample if and only if the corresponding divisor on the normalization is semi-ample.

It remains to prove Theorem~\ref{thm: equivariant sklt}. Its proof proceeds by induction on the dimension and does not require $T$ to be projective. By assumption, the normalization pair $(\Xbar,D_{\Xbar}+\Delta_{\Xbar})$ admits a log canonical model $(\Xbar^c,D_{\Xbar^c}+\Delta_{\Xbar^c})$ over $T$, hence also a good minimal model $(\Xbar^m,D_{\Xbar^m}+\Delta_{\Xbar^m})$ (see~\cite{Fujino2011_MMPtermination} and~\cite[Corollary 2.9]{Hacon_Xu_dslt}). Adjunction gives a restriction homomorphism between log canonical rings $R_{\Xbar^m}\rightarrow R_{D_{\Xbar^m}}$. By Kawamata-Viehweg vanishing, this map induces a surjection between suitable Veronese subalgebras (Lemma~\ref{lem: surjectivity restriction minimal model}). It follows that $D_{\Xbar^c}$ is the stable model of $D_{\Xbar^m}$ over $T$, and that the induced morphism $D_{\Xbar^m}\rightarrow D_{\Xbar^c}$ is crepant. There is an honest $G$-action on $\Xbar^c$ which preserves the pair $(D_{\Xbar^c},\Diff_{D_{\Xbar^c}}(\Delta_{\Xbar^c}))$ and induces an action on its normalization. On the normalization of $D_{\Xbar^c}$, there is an involution $\mu$ whose restriction over $T^0$ is the gluing involution induced by the normalization map $\Xbar^{0,c}\rightarrow X^{0,c}$. The construction of $\mu$ relies on~\cite[Theorem 11.40]{Kollar-moduli}, which is where we use our assumption that $X^0$ intersects all log centers lying on $D_X$.

Let $H=\langle G,\mu\rangle$ denote the subgroup of birational self-maps generated by $G$ and $\mu$. Then $H$ is finite and acts B-birationally on $(D_{\Xbar^m}^n,\Diff_{D_{\Xbar^m}^n}(\Delta_{\Xbar^m}))$ by Remark~\ref{rmk: Bir is a crepant invariant}. By the inductive hypothesis, the inclusions
\[
\begin{tikzcd}
    R_{D_{\Xbar^m}}\cap R_{D_{\Xbar^m}^n}^H\arrow[r,hookrightarrow] & R_{D_{\Xbar^m}^n}^{H}\arrow[r,hookrightarrow] & R_{D_{\Xbar^m}^n}
\end{tikzcd}
\]
are finite extensions of $\cO_T$-algebras. Since the restriction map $\res\colon R_{\Xbar^m}\rightarrow R_{D_{\Xbar^m}}$ is surjective on a suitable Veronese subalgebra, it follows that
\[
\begin{tikzcd}
    \res^{-1}(R_{D_{\Xbar^m}}\cap R_{D_{\Xbar^m}^n}^H)\arrow[r,hookrightarrow] & R_{\Xbar^m}
\end{tikzcd}
\]
is also an integral, hence finite, extension. Finally, passing from this to the finiteness of the extension $R_{\Xbar^m}\cap R_{\Xbar^m}^G\hookrightarrow R_{\Xbar^m}$ is a purely algebraic argument.

\subsection{Structure of the paper}

We set up the basic definitions and conventions in \S\ref{subsec: conventions} and \S\ref{sec: preliminaries}. In the latter, we also cover some preliminary material, mainly on semi-dlt pairs.

In \S\ref{sec: stable models}, we introduce and study the notion of good stable models of slc pairs, a slight generalization of log canonical models for lc pairs.
The main result of this section is Theorem~\ref{thm: projRX is stable model}, which, together with Proposition~\ref{prop: Y=proj RX good stable model}, characterizes good stable models as the Proj of the log canonical algebra, similarly to the classical case.

The proofs of the main theorems are carried out in \S\ref{sec: proof main theorems}. Specifically, Theorem~\ref{thm: equivariant sklt} is proved in \S\ref{subsec: proof equivariant sklt}, Theorem~\ref{thm: general version} in \S\ref{subsec: proof main theorem}, and Corollaries~\ref{cor: main intro} and \ref{cor: simple normal crossing case} in \S\ref{subsec:corollaries}.

Finally, in \S\ref{sec:slc-MMP} we prove Corollary~\ref{cor: analog Birkar} and discuss how to obtain the existence of certain steps of the MMP for slc pairs.

\subsection{Notation and Conventions}\label{subsec: conventions}

Recall that a scheme $X$ over $\CC$ is \emph{essentially of finite type} if it admits a finite open cover by spectra of localization of finitely generated $\CC$-algebras. We will always work with reduced schemes that are separated and essentially of finite type over $\CC$. Unless explicitly stated otherwise, all schemes are assumed to satisfy the above conditions. We will call a scheme \emph{variety} when it is further assumed to be of finite type.

\subsubsection{Birational maps of schemes}\label{subsubsec: birational maps}

Given two schemes $X$ and $Y$, we say that a rational map $f\colon X\dashrightarrow Y$ is \emph{birational} if it induces a bijection between the sets of irreducible components, and its restriction to each irreducible component is birational onto its image.

The \emph{exceptional locus} $\Ex(f)$ of a birational map $f\colon X\dashrightarrow Y$ of schemes is the complement of the largest open of $X$ on which $f$ is a morphism and an isomorphism onto its image. In particular, the indeterminacy locus of $f$ is contained in $\Ex(f)$.

\subsubsection{Relative log canonical algebra}\label{subsubsec: log canonical algebra}

Let $f\colon X\rightarrow T$ be a proper morphism and let $D$ be a $\QQ$-Cartier $\QQ$-divisor on $X$. For every $m\geq0$, we denote by $\cO_X(mD)$ the reflexive sheaf $\cO_X(\lfloor mD\rfloor)$, and set
\[
    R(X/T,D):=\oplus_{m\geq0}f_*\cO_X(mD),
\]
which has a natural structure of $\cO_T$-algebra.
When $D$ and $T$ are clear from context, we simply write $R_X$ in place of $R(X/T,D)$.

Given an integer $\ell>0$, a scheme $T$, and a sheaf of graded $\cO_T$-algebras $R=\oplus_{n\geq0} R_n$, we denote by $R^{(\ell)}:=\oplus_{n\geq0}R_{\ell n}$ the $\ell$-th \emph{Veronese subalgebra} of $R$.

\subsubsection{Discrepancies and log (canonical) centers}\label{subsubsec: conventions discrepancy and centers}

Given a demi-normal pair $(X,\Delta_X)$ (Definition~\ref{def: deminormal pairs}), a proper birational morphism $Y\rightarrow X$ and an irreducible divisor $E$ on $Y$ whose general point is regular in $Y$, we denote by $a(E,X,\Delta_X)$ the \emph{discrepancy of $E$}.

For an irreducible subscheme $W\subset X$, the \emph{minimal log discrepancy} of $W$ is
\[
    \mathrm{mld}(W,X,\Delta_X)=\mathrm{inf}_{E}\{1+a(E,X,\Delta_X) : \mathrm{center}_{X}E=W\}.
\]
Note that in the above definition we used \emph{log discrepancies}, as this seems to be the usual variant here.

A \emph{log center} $W$ of a log canonical pair $(X,\Delta_X)$ is the center of a divisor over $X$ with negative discrepancy; equivalently, $\mathrm{mld}(W,X,\Delta_X)<1$. It is called a \emph{log canonical center} if $\mathrm{mld}(W,X,\Delta_X)=0$. We also regard each connected component of $X$ as a log canonical center. These definitions extend to slc pairs $(X,\Delta_X)$, by declaring a log center (respectively, log canonical center) to be the image of a log center (respectively, log canonical center) of its normalization pair $(\Xbar,D_{\Xbar}+\Delta_{\Xbar})$. Note that this includes the irreducible components of $X$ and the strata of the conductor $D_X$. See~\cite[\S4-5]{Kollar-singularities} for more details.

\subsection*{Acknowledgments}

This project was motivated by a question raised by Dan Abramovich about simplifying the proof of the properness of the KSBA moduli space of stable pairs. The author is grateful to him for his suggestions and support. He would also like to thank Alessio Corti, Christopher Hacon, and J\'anos Koll\'ar for generously answering numerous questions and sharing their insights. Finally, the author thanks Eric Jovinelly, Roberta Pagliaro, Simon Stojkovic, and Hao (Nick) Sun for useful conversations on this and related projects.
This research is supported by NSF grant DMS-2401358.

\subsection*{AI disclosure}

The author used AI assistance in a limited capacity for proofreading, improving the clarity and exposition of this paper, and occasional review of proofs. This review did not result in substantive modifications to the proofs. The mathematical results and arguments were developed independently by the author.

\section{Preliminaries}\label{sec: preliminaries}

In this section we recall some basic definitions and properties that we will need in the following. Most of the notions and results can be found in~\cite[\S5]{Kollar-singularities}.

\begin{definition}[Demi-normal pairs]\label{def: deminormal pairs}
    A \emph{scheme} will always be assumed to be reduced, separated, and essentially of finite type over $\CC$. Recall that a scheme $X$ over $\CC$ is essentially of finite type if it admits a finite open cover by spectra of localizations of finitely generated $\CC$-algebras. A \emph{variety} is a scheme of finite type. Let $X$ be a scheme.
    \begin{itemize}
        \item We say that $X$ is \emph{demi-normal} if it has at most nodal singularities in codimension 1, and it satisfies Serre's property $S_2$. The \emph{conductor} of $X$ is the closure of the nodal locus of $X$, and denoted by $D_X$ (\cite[5.2]{Kollar-singularities}).
        \item A \emph{Weil divisor} $\Delta_X=\sum_i a_iD_i$ on $X$ is a finite sum with integral coefficients $a_i$ of irreducible codimension-1 subschemes $D_i$ such that $X$ is regular at their generic points. If we allow the coefficients $a_i$ to be rational, we say that $\Delta_X$ is a \emph{$\QQ$-divisor}. A $\QQ$-divisor is a \emph{boundary} if the coefficients satisfy $0\leq a_i\leq 1$.
        \item A \emph{$\QQ$-Cartier $\QQ$-divisor} on a demi-normal scheme is a $\QQ$-divisor $\Delta_X$ for which there exists a positive integer $m$ such that $m\Delta_X$ is integral and Cartier.
        \item A pair $(X,\Delta_X)$ is a \emph{demi-normal pair} (respectively, \emph{normal pair}) if $X$ is a demi-normal scheme (respectively, normal), and $\Delta_X$ is a boundary $\QQ$-divisor in $X$.
    \end{itemize}
\end{definition}

\begin{remark}
    Let $X$ be a demi-normal scheme, and let $\Delta_X=\sum_i a_iD_i$ be sum of irreducible codimension-1 subschemes. Then, $\Delta_X$ is a $\QQ$-divisor if and only if for all $i$ the subscheme $D_i$ is not an irreducible component of the conductor $D_X$.
\end{remark}

\begin{remark}[Normalization pair]\label{rmk: normalization pair}
    Let $(X,\Delta_X)$ be a demi-normal pair, and $\pi_X\colon \Xbar\rightarrow X$ the normalization. The ideal sheaf of the conductor $D_X\subset X$ is the largest ideal sheaf of $\cO_X$ that is also an ideal sheaf of $\cO_{\Xbar}$, and it is equal to $\mathcal{H}om_X(\pi_*\cO_{\Xbar},\cO_X)$. The corresponding divisor in $\Xbar$ is also called conductor and denoted by $D_{\Xbar}$.
    
    The \emph{normalization pair} is $(\Xbar,D_{\Xbar}+\Delta_{\Xbar})$ where $\Delta_{\Xbar}$ is the strict transform of $\Delta_X$.

    We denote by $D_{\Xbar}^n$ the normalization of $D_{\Xbar}$. The map $D_{\Xbar}\rightarrow D_X$ between conductors is generically a double cover, and the associated birational involution on $D_{\Xbar}$ extends to a regular involution $\mu\colon D_{\Xbar}^n\rightarrow D_{\Xbar}^n$. Moreover, $\mu$ is compatible with the different (\S\ref{subsec: adjunction}): $\mu^*\Diff_{D_{\Xbar}^n}(\Delta_{\Xbar})=\Diff_{D_{\Xbar}^n}(\Delta_{\Xbar})$, see~\cite[Proposition 5.12]{Kollar-singularities}. The pair $(X,\Delta_X)$ is uniquely determined by the tuple $(\Xbar,D_{\Xbar},\Delta_{\Xbar},\mu)$ (\cite[Proposition 5.3]{Kollar-singularities}). In particular, two demi-normal pairs that are isomorphic in codimension 1 and have the same normalization are necessarily isomorphic.

    Given a tuple $(\Xbar,D_{\Xbar}+\Delta_{\Xbar},\mu)$, deciding whether or not this arises as the normalization pair of a demi-normal pair is incredibly hard and not fully understood. Koll\'ar's gluing theory \cite[\S5, \S9]{Kollar-singularities} was designed precisely to give a satisfactory answer to this problem in the context of slc pairs, and proved to be very successful.
\end{remark}

The following construction is the demi-normal analogue of the normalization of a scheme; see~\cite[Definition 5.1]{Kollar-singularities}.

\begin{definition}[Demi-normal modification]\label{def: deminormal modification}
    Let $Z$ be a scheme. Suppose that there exists an open subscheme $j\colon U\hookrightarrow Z$ with at most nodal singularities, and whose complement has codimension at least 2. Then, $\widetilde{Z}:=\Spec_Z j_*\cO_U$ is a demi-normal scheme with a finite morphism $\delta_Z\colon \widetilde{Z}\rightarrow Z$, called the \emph{demi-normal modification} of $Z$.
\end{definition}

The following simple lemma shows that the demi-normal modification is independent of the open $U$, and satisfies a universal property.

\begin{lemma}\label{lem: universal property deminormal modification}
    With the notation of Definition~\ref{def: deminormal modification}, the following hold.
    \begin{enumerate}
        \item $\delta_Z$ is an isomorphism over $U$, and the normalizations of $\widetilde{Z}$ and $Z$ coincide.
        \item Let $f\colon X\rightarrow Z$ be a morphism from a demi-normal scheme, and suppose that $f^{-1}(U)\subset X$ contains all the generic points of the conductor of $X$. Then, $f$ factors uniquely through $\delta_Z$. Moreover, this property is satisfied if $f$ does not contract any component of the conductor. In particular, $\delta_Z\colon \widetilde{Z}\rightarrow Z$ is independent of $U$.
    \end{enumerate}
\end{lemma}
\begin{proof}
    The first point follows immediately from the definition. For the second point, note that giving a lift $X\rightarrow\widetilde{Z}$ of $f$ is equivalent to giving a factorization $\cO_Z\rightarrow j_*\cO_U\rightarrow f_*\cO_X$. Since $V=X\setminus D_X$ is normal, $f|_V\colon V\rightarrow Z$ factors through the normalization of $Z$, hence through $\widetilde{Z}$ by the first point of this lemma. Let $i\colon W=f^{-1}(U)\cup V\hookrightarrow X$, and note that its complement has codimension at least 2 in $X$. Since $X$ and $Z$ are separated, any factorization through $\widetilde{Z}$ is unique. In particular, $f|_W\colon W\rightarrow Z$ factors through $\delta_Z$. Since $X$ is demi-normal, we get the desired morphism $j_*\cO_U\rightarrow (f|_W)_*\cO_W=f_*i_*\cO_W\simeq f_*\cO_X$.
\end{proof}

\begin{example}\label{exm: deminormal modifications}
    We present some examples where the demi-normal modification is useful.
    \begin{enumerate}
    \item\label{exm: deminormal stein factorization} Let $X\rightarrow Y$ be a proper morphism from a demi-normal scheme $X$ to a scheme $Y$ that is nodal in codimension 1, and let $X\rightarrow Z\rightarrow Y$ be the Stein factorization. Suppose that $X\rightarrow Y$ has connected fibers over an open $U\subset Y$ whose complement has codimension at least 2, and $f$ does not contract components of the conductor. Then, $Z\rightarrow Y$ is an isomorphism over $U$, hence $Z$ is nodal in codimension 1. Since $X$ is demi-normal, $X\rightarrow Z$ factors through the demi-normal modification $\widetilde{Z}\rightarrow Z$, which is a finite map. From the properties of the Stein factorization, we get that $\widetilde{Z}\simeq Z$, that is, $Z$ is demi-normal.
    \item\label{exm: deminormal boundary} Let $X$ be a normal scheme and $S\subset X$ a reduced subscheme of pure codimension 1 such that $(X,S+B)$ is an lc pair for some $\QQ$-divisor $B$. Then, $S$ is nodal in codimension 1 by~\cite[Proposition 16.6]{K+92}, hence we can take the universal demi-normal modification $\widetilde{S}\rightarrow S$. In general $S$ does not satisfy Serre's condition $S_2$, even when $S$ is realized as the conductor of a normalization pair. See~\cite[Example 17.5.2]{K+92} or Example~\eqref{exm: conductor is not always S2} for instances where the $S_2$ property fails.
    \item \label{exm: conductor is not always S2}
    Consider the product $S=E\times\PP^1$ where $E$ is a smooth curve of genus 1. Let $L$ be an ample line bundle on $S$, and let $X$ be the cone over $S$ corresponding to $L$. Its natural resolution is the total space of $L$ itself. Let $D=E\times\{0,\infty\}$, and $\Delta_X$ be the cone over $D$ corresponding to $L|_D$. Since $K_S+D\sim 0$, the pair $(X,\Delta_X)$ is an lc pair that is not dlt by~\cite[Lemma 3.1]{Kollar-singularities}. In fact, $\Delta_X=\lfloor\Delta_X\rfloor$ is the union of two surfaces meeting at one point, hence not $S_2$.
    \end{enumerate}
\end{example}

We remind the reader of the definition of dlt and semi-dlt pairs, and introduce the notion of semi-klt pair. First, let us recall the definition of \emph{semi-SNC} of~\cite[Definition 1.10]{Kollar-singularities}; see~\cite[\S2.1]{Wlodarczyk-nc} for related notions. As usual, SNC stands for \emph{simple normal crossing}.

\begin{definition}[semi-SNC pairs]\label{def: SNC adapted}
    Let $X$ be a SNC scheme. We say that a demi-normal pair $(X,\Delta_X)$ is \emph{semi-SNC} if Zariski-locally we can embed $X$ into a smooth scheme $Y$ carrying a SNC divisor $X+\Delta_Y$ such that $\Delta_Y|_X=\Delta_X$. We also say that $\Delta_X$ is \emph{SNC-adapted}.
\end{definition}

\begin{definition}[dlt/sdlt/sklt pairs]\label{def: dlt/sdlt/sklt}
    A pair $(X,\Delta_X)$ with $X$ normal is dlt if there exists a closed subscheme $Z\subset X$ such that $U:=X\setmin Z$ is smooth, $\Delta_X|_U$ is a SNC divisor on $U$, and $a(E,X,\Delta_X)>-1$ for all divisors $E$ over $X$ with center in $Z$.

    An slc pair $(X,\Delta_X)$ is \emph{semi-divisorial log terminal}, or \emph{semi-dlt}, if there exists a closed subscheme $Z\subset X$ such that $(U,\Delta_U)$ is semi-SNC, where $U:=X\setmin Z$ and $\Delta_U:=\Delta_X|_U$, while $a(E,X,\Delta_X)>-1$ for all divisors $E$ over $X$ with center in $Z$.

    A semi-dlt pair $(X,\Delta_X)$ is called \emph{semi-Kawamata log terminal}, or \emph{semi-klt}, if in addition $\lfloor\Delta_X\rfloor=0$.
\end{definition}

\begin{remark}\label{rmk: basics semidlt}
    Note that every irreducible component of a semi-dlt pair $(X,\Delta_X)$ is normal \cite[Proposition 5.20]{Kollar-singularities}, and the normalization pair $(\Xbar,D_{\Xbar}+\Delta_{\Xbar})$ is dlt.
    
    In other works, semi-dlt pairs are more generally defined as slc pairs whose normalization is dlt and with normal irreducible components (\cite[Definition 1.1]{FujinoAbundanceSLC},\cite[Definition 2.5]{FG_fin_Brepr}). We instead adopt the more stringent version of K\'ollar~\cite[Definition 5.19]{Kollar-singularities}.
    
    The definition of semi-klt pair is motivated by the fact that a dlt pair $(Y,\Delta_Y)$ is klt if and only if $\lfloor\Delta_Y\rfloor=0$ (\cite[Proposition 2.41]{Kollar_Mori_1998}). In particular, a semi-dlt pair $(X,\Delta_X)$ is semi-klt if and only if $(\Xbar,\Delta_{\Xbar})$ is klt.
\end{remark}

\begin{remark}\label{rmk: self-intersection case}
    Recall that every lc pair admits a $\QQ$-factorial crepant dlt modification with desirable properties~\cite[\S1.4]{Kollar-singularities}. For this reason, results about dlt pairs have fundamental consequences for all lc pairs.
    In contrast, semi-dlt pairs are too special to reduce statements about slc pairs to the semi-dlt case without any additional argument, as they have normal irreducible components. However, for many applications, such as finite generation of the log canonical algebra, we are allowed to take finite covers. We can get rid of codimension-1 self-intersections by taking a (Galois) double cover that is étale in codimension 1 and trivial at the level of normalizations; see~\cite[5.23]{Kollar-singularities} for details. Then, for an slc pair whose irreducible components are normal in codimension 1 we can apply~\cite[Theorem 1.2]{Hashizume_sdlt_models} to get a Gorenstein crepant semi-dlt model (see also \S\ref{sec:slc-MMP}).
\end{remark}

\subsection{Adjunction}\label{subsec: adjunction}

One of the most important tools in the Minimal Model Program is \emph{adjunction}. Given a normal pair $(X,S+B)$ with $\lfloor S\rfloor=S$, there exists a natural divisor $\Diff_{S^n}(B)$ on the normalization $S^n$ of $S$, called \emph{different}, such that $(S^n,\Diff_{S^n}(B))$ is a normal pair and a variant of Poincar\'e Residue Theorem applies. Moreover, the singularity behavior of $(X,S+B)$ near $S$ and that of $(S^n,\Diff_{S^n}(B))$ are tightly linked, and one is lc if and only if the other is. Results of this type are called adjunction and inverse of adjunction, and are essential for inductive proofs.
When $S$ satisfies Serre's condition $S_2$, the different can be defined directly on $S$; more generally, it can be defined on the demi-normal modification $\widetilde{S}$ of it (Definition~\ref{def: deminormal modification}).
We refer to~\cite[\S4]{Kollar-singularities} and~\cite[\S16]{K+92} for the rigorous constructions and results.

In this subsection we limit ourselves to collecting the following results on adjunction for (semi-)dlt pairs.

\begin{lemma}\label{lem: semidlt adjunction}
    Let $(X,\Delta_X)$ be a dlt pair, $S\leq\lfloor\Delta_X\rfloor$, and $B=\Delta_X-S$. Then,
    \begin{enumerate}
        \item $S$ is demi-normal, and $(S,\Diff_{S}(B))$ is semi-dlt.
        \item If $S=\lfloor\Delta_X\rfloor$ (equivalently, $\lfloor B\rfloor=0$), then $(S,\Diff_{S}(B))$ is semi-klt.
    \end{enumerate}
    In particular, if $(Z,\Delta_Z)$ is a semi-dlt pair and $(\Zbar,D_{\Zbar}+\Delta_{\Zbar})$ is the normalization pair, then $(D_{\Zbar},\Diff_{D_{\Zbar}}(\Delta_{\Zbar}))$ is semi-dlt. Moreover, $(D_{\Zbar},\Diff_{D_{\Zbar}}(\Delta_{\Zbar}))$ is semi-klt if and only if $(Z,\Delta_Z)$ is semi-klt along its conductor $D_Z$.
\end{lemma}
\begin{proof}
    The first point is well-known, see for instance~\cite[Definition 5.19]{Kollar-singularities}.
    
    For the second point, by localizing at codimension 1 points of $S$, we can assume $(X,S+B)$ to be a surface dlt pair, hence $\QQ$-factorial by~\cite[Proposition 4.11]{Kollar_Mori_1998}. By $\QQ$-factoriality, we have that $\Diff_{S}(B)=\Diff_{S}(0)+B|_S$, and by the first point it is enough to show that $\lfloor\Diff_S(B)\rfloor=0$. By~\cite[Theorem 3.36]{Kollar-singularities}, if $p\in S$ is a point at which $\Diff_S(0)$ has coefficient 1, then $S$ is smooth at $p$. Then, $(X,S+B)$ is plt locally around $p$ by~\cite[Proposition 2.42]{Kollar_Mori_1998}, hence $(S^n,\Diff_{S^n}(B))$ is klt around $p$ by adjunction (\cite[Theorem 5.50]{Kollar_Mori_1998}); in particular, the coefficient of $\Diff_{S}(0)$ is actually strictly smaller than 1. Now, let $p$ be in the support of $B|_S$. Since we are working with surfaces, the characterization of lc centers on dlt pairs implies that $p$ is again a smooth point of $S$. It follows again by adjunction (\cite[Theorem 5.50]{Kollar_Mori_1998}) that the boundary pair is klt.

    Finally, let $(Z,\Delta_Z)$ be a semi-dlt pair. Then, its normalization pair is dlt, and $(\Zbar,\Delta_{\Zbar})$ is klt if and only if $(Z,\Delta_Z)$ is semi-klt (Remark~\ref{rmk: basics semidlt}). Then, the result follows from the first two points of the statement.
\end{proof}

We conclude the subsection by recalling the description of the log canonical ring of an slc pair in terms of the log canonical ring of its normalization and the gluing involution on the conductor (Remark~\ref{rmk: normalization pair}). See~\cite[Proposition 5.8]{Kollar-singularities} for a proof for the case over a field.

\begin{lemma}\label{lem: log canonical ring slc pair}
    Let $(X,\Delta_X)$ be an slc pair with normalization pair $(\Xbar,D_{\Xbar}+\Delta_{\Xbar})$, where $D_{\Xbar}$ is the conductor, and let $f\colon X\rightarrow T$ be a proper morphism to a scheme $T$. Let $\mu$ be the gluing involution on $(D_{\Xbar}^n,\Diff_{D_{\Xbar}^n}(\Delta_{\Xbar}))$, and let $R_{D_{\Xbar}^n}:=R(D_{\Xbar}^n/T,K_{D_{\Xbar}^n}+\Diff_{D_{\Xbar}^n}(\Delta_{\Xbar}))$. Let $\mu$ act on $R_{D_{\Xbar}^n}$ via $\mu\cdot\alpha=(-1)^m\mu^*\alpha$ for all $\alpha\in R_{D_{\Xbar}^n}$ of degree $m$. Then, the inclusion
    \[
    \begin{tikzcd}
        R_X:=R(X/T,K_X+\Delta_X)\arrow[r,hookrightarrow] & R(\Xbar/T,K_{\Xbar}+D_{\Xbar}+\Delta_{\Xbar})=:R_{\Xbar}
    \end{tikzcd}
    \]
    yields an identification
    \[
        R_X=\res_n^{-1}(R_{D_{\Xbar}^n}^\mu),
    \]
    where $R_{D_{\Xbar}^n}^{\mu}$ denotes the subalgebra of $\mu$-invariants, and
    \[
    \begin{tikzcd}
        \res_n\colon R_{\Xbar}\arrow[r] & R_{D_{\Xbar}^n}
    \end{tikzcd}
    \]
    is the natural residue (or restriction) map.
\end{lemma}

\begin{remark}
    The residue map $\res_n$ factors through $\res\colon R_{\Xbar}\rightarrow R_{\widetilde{D}_{\Xbar}}$, where $\widetilde{D}_{\Xbar}$ is the demi-normal modification of $D_{\Xbar}$ (see Definition~\ref{def: deminormal modification} and Example~\ref{exm: deminormal modifications}\eqref{exm: deminormal boundary}) and $R_{\widetilde{D}_{\Xbar}}$ the corresponding log canonical $\cO_T$-algebra. In particular, if $(X,\Delta_X)$ is semi-dlt then $\res_n$ factors through $\res\colon R_{\Xbar}\rightarrow R_{D_{\Xbar}}$.
\end{remark}

\subsection{B-Birational Maps}

We recall the basic notion of \emph{B-birational maps}; here, the `B' stands for `boundary', as they can be thought as birational maps preserving the boundary. These have proved to be of crucial importance when dealing with slc pairs. For instance, the finiteness of the associated B-representations represents one of the fundamental results in the theory~\cite{FujinoAbundanceSLC,FG_fin_Brepr,Hacon_Xu_finBrepr_slc_abundance}, and led to the reduction of the abundance conjecture for slc pairs to the case of log canonical pairs.

\begin{definition}[Relative B-birational maps, {\cite[Definition 3.1]{FujinoAbundanceSLC}}]\label{def: B-birational maps}
    Let $(X,\Delta_X)$ and $(Y,\Delta_Y)$ be pairs of normal schemes together with $\QQ$-divisors, such that $K_X+\Delta_X$ and $K_Y+\Delta_Y$ are $\QQ$-Cartier. Let $f\colon X\rightarrow T$ and $g\colon Y\rightarrow T$ be proper morphisms to a scheme $T$. We say that a birational map $\sigma\colon (X,\Delta_X)\dashrightarrow(Y,\Delta_Y)$ over $T$ is \emph{B-birational} if there exists a common resolution
    \[
    \begin{tikzcd}
        & W\arrow[ld,"\alpha"']\arrow[rd,"\beta"]\\
        X\arrow[rr,dashed,"\sigma"] & & Y
    \end{tikzcd}
    \]
    such that
    \[
        \alpha^*(K_X+\Delta_X)=\beta^*(K_Y+\Delta_Y).
    \]
    We denote by $\Bir(X/T,\Delta_X)$ the group of B-birational self-maps of $(X,\Delta_X)$ over $T$, where the group operation is the composition.
\end{definition}

\begin{remark}\label{rmk: Bir is a crepant invariant}
    An important property of $\Bir(X,\Delta_X)$ is that it is a crepant birational invariant, in the following sense. Let $\psi\colon (X,\Delta_X)\rightarrow(Y,\Delta_Y)$ be a proper crepant birational morphism between normal pairs proper over $T$, then $\Bir(Y/T,\Delta_Y)\simeq\Bir(X/T,\Delta_X)$ via $g\mapsto\psi^{-1}\circ g\circ\psi$; see~\cite[Remark 2.12]{FG_fin_Brepr}.
    We will apply this observation when $(X,\Delta_X)$ is a minimal model over $T$, and $Y$ is its relative log canonical model.
\end{remark}

Another property of B-birational self-maps we will use is that they act on the spaces of pluri-canonical sections, thus on a suitable Veronese subalgebra of the log canonical ring. This follows immediately from the definition of B-birational self-maps over $T$, hence we omit the proof of the following lemma; see for instance~\cite[Definition 2.14]{FG_fin_Brepr}.

\begin{lemma}\label{lem: Bir acts on log canonical}
    Let $X$ be a normal scheme and $f\colon X\rightarrow T$ a proper morphism. Let $\Delta_X$ be a $\QQ$-divisor, and let $m$ be a non-negative integer such that $m(K_X+\Delta_X)$ is Cartier. Then, any B-birational map $\sigma\in\Bir(X/T,\Delta_X)$ over $T$ naturally induces a linear automorphism of $f_*\cO_{X}(m(K_X+\Delta_X))$.
    
    In particular, if $\ell$ is the smallest positive integer such that $\ell(K_X+\Delta_X)$ is Cartier, then $\Bir(X/T,\Delta_X)$ acts on $R(X/T,K_X+\Delta_X)^{(\ell)}$.
\end{lemma}

\begin{remark}\label{rmk: different invariants mu}
    Usually, we care about properties of $R_X=R(X/T,K_X+\Delta_X)$ that hold if and only if they are satisfied by some Veronese subalgebra, such as finite generation. Thus, the fact that $\Bir(X/T,\Delta_X)$ acts only on a Veronese subalgebra of $R_X$ does not represent an issue.
    However, there is a slight misalignment between the action defined in Lemma~\ref{lem: Bir acts on log canonical} and that in Lemma~\ref{lem: log canonical ring slc pair} in the case where $G=\langle\mu\rangle$ is the gluing involution associated to an slc pair. Indeed, in Lemma~\ref{lem: log canonical ring slc pair} the action is twisted by a factor of $(-1)^m$ in degree $m$. Hence the two actions differ only in odd degrees. By restricting to a Veronese subalgebra of even degree $\ell$, the twisting disappears. We will thus always take $\ell$ to be even, which justifies the following definition. Note also that we use different notation for the $\mu$-invariants under the ``twisted'' action, as $\mu$ is an element of $\langle\mu\rangle$ rather than the group itself.
\end{remark}

\begin{definition}\label{def: invariants under Bir}
    Let $(X,\Delta_X)$ be an lc pair, $f\colon X\rightarrow T$ a proper morphism, and $R_X:=R(X/T,K_X+\Delta_X)$. Given a finite subgroup $G\subseteq\Bir(X/T,\Delta_X)$, we denote by $R_X^G$ the subalgebra of $R_X^{(\ell)}$ of $G$-invariants, where $\ell$ is the smallest positive \emph{even} integer such that $\ell(K_X+\Delta_X)$ is Cartier.
\end{definition}

Since we only care about finite generation of algebras, we will not keep track of $\ell$ in the paper, in particular in Theorem~\ref{thm: equivariant sklt} and its proof.

\section{Good stable models of semi-log-canonical pairs}\label{sec: stable models}

In this section we introduce a notion of \emph{stable model} for slc pairs, slightly extending log canonical models to the non-normal case. We limit ourselves to treating the case where the exceptional locus of the map to the stable model does not contain any component of the conductor; for this reason, we call them \emph{good stable models}.
The general case is much more subtle, see for instance Example~\ref{exm: Proj RX not deminormal}.
We also provide examples of maps that should reasonably be considered stable models but fail to be good stable models, showing that the requirements~\eqref{property: defined up to codim 2} and~\eqref{property: birational conductor} in Definition~\ref{def: stable model slc pairs} are too restrictive in general.

\subsection{Definition and examples of good stable models}

A stable model is required to satisfy a large amount of properties, many of which will be proven to be redundant.

\begin{definition}[Stable Models]\label{def: stable model slc pairs}
    Let $(X,\Delta_X)$ be an slc pair, and let $(Y,\Delta_Y)$ be a demi-normal pair with $\QQ$-Cartier log canonical divisor $K_Y+\Delta_Y$. Let $X\rightarrow T$ be a proper morphism, and let
    \[
    \begin{tikzcd}
        \phi\colon X\arrow[r,dashed] & Y
    \end{tikzcd}
    \]
    be a rational map over $T$. We say that $\phi$ is a \emph{good stable model over $T$} if it satisfies the following properties:
    \begin{enumerate}
        \item\label{property: Y slc} $(Y,\Delta_Y)$ is slc;
        \item\label{property: ample} $K_Y+\Delta_Y$ is ample over $T$;
        \item\label{property: defined up to codim 2} $\phi$ is defined on an open $U\subset X$ with complement of codimension at least 2;
        \item\label{property: birational contraction} $\phi$ is a birational contraction, that is, $\mathrm{codim}_Y(\Ex(\phi^{-1}))\geq2$;
        \item\label{property: normalization lc model} the induced map between normalization pairs
        \[
        \begin{tikzcd}
            \phibar\colon (\Xbar,D_{\Xbar}+\Delta_{\Xbar})\arrow[r] & (\Ybar,D_{\Ybar}+\Delta_{\Ybar})
        \end{tikzcd}
        \]
        is a log canonical model over $T$;
        \item\label{property: pushforward normalization} $\phibar_*D_{\Xbar}=D_{\Ybar}$, and $\phibar_*\Delta_{\Xbar}=\Delta_{\Ybar}$;
        \item\label{property: pushforward} $\phi_*\Delta_X=\Delta_Y$.
        \item\label{property: birational conductor} the restriction of $\phi$ to $D_X$ defines a birational map $\phi|_{D_X}\colon D_X\dashrightarrow D_Y$. 
    \end{enumerate}
    When $X$ is normal, we call the stable model simply a log canonical model, as the two notions coincide.
\end{definition}

Note that the dependence on the base $T$ appears only in Properties~\eqref{property: ample} and~\eqref{property: normalization lc model}.

\begin{lemma}\label{lem: redundant def stable model}
    Let $(X,\Delta_X)$ be an slc pair proper over a scheme $T$, and let $(Y,\Delta_Y)$ be a demi-normal pair over $T$ with $\QQ$-Cartier log canonical divisor $K_Y+\Delta_Y$. Let $\phi\colon X\dashrightarrow Y$ be a rational map over $T$ defined on an open subscheme of $X$ with complement of codimension at least 2. We refer to Definition~\ref{def: stable model slc pairs} for the properties.
    \begin{enumerate}[label=$(\alph*)$]
        \item\label{point: equivalence 1 and 2} Property~\eqref{property: pushforward normalization} implies~\eqref{property: pushforward}, and the converse holds if we assume~\eqref{property: birational contraction},~\eqref{property: normalization lc model} or~\eqref{property: birational conductor}.
        \item Property~\eqref{property: normalization lc model} implies~\eqref{property: Y slc} and~\eqref{property: ample}. If we assume property~\eqref{property: pushforward normalization}, then also~\eqref{property: birational contraction} follows.
        \item If the properties~\eqref{property: normalization lc model} and~\eqref{property: birational conductor} hold, then $\phi$ is a good stable model.
        \item If $\phi$ is a good stable model, then for every divisor $E$ over $X$, we have that
        \begin{equation}\label{eq: inequality discrepancies}
            a(E,X,\Delta_X)\leq a(E,Y,\Delta_Y).
        \end{equation}
        Conversely, if inequality \eqref{eq: inequality discrepancies} holds for the $\phi$-exceptional divisors and properties~\eqref{property: Y slc}, \eqref{property: ample}, \eqref{property: birational contraction}, \eqref{property: pushforward} and \eqref{property: birational conductor} are satisfied, then $\phi$ is a good stable model.
    \end{enumerate}
\end{lemma}
\begin{proof}
    We prove each point separately.
    \begin{enumerate}[label=$(\alph*)$]
        \item The first part follows from the fact that $\Delta_X$ and $\Delta_Y$ are birational to $\Delta_{\Xbar}$ and $\Delta_{\Ybar}$, respectively. For the second part, note that Property~\eqref{property: normalization lc model} implies $\phibar_*(D_{\Xbar}+\Delta_{\Xbar})=D_{\Ybar}+\Delta_{\Ybar}$, while Property~\eqref{property: birational conductor} implies $\phibar_*D_{\Xbar}=D_{\Ybar}$ directly. If~\eqref{property: birational contraction} holds, then $\phi^{-1}$ is an isomorphism at any codimension 1 point, hence $\phi_*^{-1}D_Y\subset D_X$ and $\phi_*D_X=D_Y$, hence again $\phibar_*D_{\Xbar}=D_{\Ybar}$.
        \item Properties~\eqref{property: Y slc} and~\eqref{property: ample}  follow immediately by~\cite[Definition-Lemma 5.10]{Kollar-singularities} and the fact that a Cartier divisor on $Y$ is relatively ample if and only if its pullback to $\Ybar$ is relatively ample. Property~\eqref{property: birational contraction} follows from the fact that the normalization is finite and $\phi_*D_{\Xbar}=D_{\Ybar}$, so that no component of $D_Y$ is in $\Ex(\phi^{-1})$.
        \item By the previous point, it is enough to show that under these assumptions~\eqref{property: pushforward normalization} and~\eqref{property: pushforward} are satisfied. By the first point of this lemma, it is enough to show that $\phibar_*D_{\Xbar}=D_{\Ybar}$, which in turn follows from the birationality of $\phi|_{D_X}\colon D_X\dashrightarrow D_Y$.
        \item The inequality between discrepancies follows from the analogous one at the level of normalization pairs \cite[Definition 3.50, Proposition 3.51]{Kollar_Mori_1998}. The converse holds by definition of log canonical model, see loc.\! cit.
    \end{enumerate}
\end{proof}

By Lemma~\ref{lem: redundant def stable model}, a more `efficient' set of properties characterizing stable models are~\eqref{property: defined up to codim 2}, \eqref{property: normalization lc model}, and~\eqref{property: birational conductor}. It can be convenient to define good stable models intrinsically without referring to the normalization, in which case Lemma~\ref{lem: redundant def stable model} shows that we can replace the requirement of $\phibar$ being a log canonical model with the inequality~\eqref{eq: inequality discrepancies} for $\phi$-exceptional divisors, and disregard \eqref{property: pushforward normalization}. Contrarily to what happens for normal schemes, the property of $\phi$ being regular on an open subset $U$ with complement of codimension at least 2 is not automatically satisfied, as $\phi$ could fail to be regular at the generic points of the conductor. Note that the birationality of $D_X\dashrightarrow D_Y$ implicitly requires $\phi$ to be defined at the generic points of $D_X$.

\begin{example}\label{exm: bad conductor maps}
    Let us present some examples of morphisms $\phi$ that should reasonably be considered to be stable models even if they fail to be \emph{good} stable models. The examples will fail properties~\eqref{property: defined up to codim 2} or~\eqref{property: birational conductor}. Here, the base scheme is $T=\Spec\CC$.
    \begin{enumerate}
        \item\label{not defined codim 1 example 1} We present a simple example where $\phi$ is not defined at the generic point of the conductor. Let $S$ be a smooth surface with $K_S$ ample. Let $p\in S$, and $\Bl_{p}S$ the blowup of $S$ along $p$, with exceptional $(-1)$-curve $E$. Let $X=\Bl_{p}S\cup_{E}\Bl_{p}S$ be the union of two copies of $S_1$ glued along $E$, thus $D_X=E$. Note that $X$ is simple normal crossing, and $K_X$ is big on $X$ as its pullback on each irreducible component $\Bl_{p}S$ is $K_{\Bl_{p}S}+E$. Note that $(K_{\Bl_{p}S}+E)|_E\simeq K_E\simeq\cO_{\PP^1}(-2)$, hence $R_X\simeq R_S\times R_S$. It follows that $R_X$ is finitely generated and $\phi\colon X\dashrightarrow \Proj R_X\simeq S\sqcup S$ satisfies all the properties of Definition~\ref{def: stable model slc pairs}, except for~\eqref{property: defined up to codim 2} and~\eqref{property: birational conductor}.
        \item\label{conductor example 1}
        Let $S$ and $p\in S$ be as in the previous example. Let $X$ be the variety obtained by gluing $\Bl_{p}S$ along an involution $\alpha$ of the exceptional divisor $E\simeq\PP^1$ with a unique fixed point.
        Note that $X$ is slc and Gorenstein, as it has only nodal singularities or pinch points along $D_X$. The induced birational morphism $\phi\colon X\rightarrow S$ automatically satisfies properties~\eqref{property: defined up to codim 2} and \eqref{property: pushforward} of Definition~\ref{def: stable model slc pairs}. Moreover, the normalization pair of $X$ is $(\Xbar=\Bl_{p}S,E)$, while $S$ is smooth, hence~\eqref{property: normalization lc model} is also satisfied. Therefore, $\phi$ has all the properties of a good stable model, except for~\eqref{property: birational conductor}.
        \item\label{conductor example 2} Let $Z$ be a smooth projective surface, and $C\subset Z$ a smooth curve with $K_Z+C$ ample. Let $Y$ be the union of two copies of $Z$ along $C$. Note that $Y$ is a stable slc pair with only SNC singularities and conductor $C$. Let $p\in C\subset Y$ and $X_0=\Bl_pY$. Then, $X_0=\Bl_pZ\cup_{C'}\Bl_pZ$, where $C'$ is the strict transform of $C$; note that $C'\simeq C$. Let $X_1$ be the scheme obtained by gluing $X_0$ along the two exceptional divisors $E\subset Z$, whose image in $X_1$ is still denoted by $E$. Then, $(X_1,0)$ is slc with conductor $E\cup C'$. Let $\phi_i\colon X_i\rightarrow Y$ be the induced morphisms. Then, a simple computation shows that both $\phi_0$ and $\phi_1$ satisfy all properties of Definition~\ref{def: stable model slc pairs} except that $\phi_1$ fails~\eqref{property: birational conductor}. Note that $\phi_1$ contracts the component $E$ of the conductor, even though it restricts surjectively to $D_{X_1}=E\cup C'\rightarrow C=D_Y$.
    \end{enumerate}
\end{example}

\begin{remark}\label{rmk: non stable models}
    Let $(X,\Delta_X)$ be a non-normal slc pair with ample log canonical divisor. Then, the normalization map $\pi_X\colon\Xbar\longrightarrow X$ is not a stable model for any boundary on $\Xbar$, as it fails properties \eqref{property: birational contraction} and \eqref{property: pushforward normalization} of Definition~\ref{def: stable model slc pairs}. Similarly, the inverse $\pi_X^{-1}\colon X\dashrightarrow\Xbar$ of the normalization map fails properties \eqref{property: birational contraction} and \eqref{property: pushforward normalization} of Definition~\ref{def: stable model slc pairs}. It also fails~\eqref{property: defined up to codim 2}, but it satisfies property~\eqref{property: pushforward}.
\end{remark}

\begin{remark}\label{rmk: general definition stable model}
    Example~\ref{exm: bad conductor maps} shows that to obtain a possible general notion of stable model we need to remove condition~\eqref{property: defined up to codim 2}, and of course~\eqref{property: birational conductor}. Remark~\ref{rmk: non stable models} suggests that all other properties are to be maintained.
    We omit a further discussion on this, as it is beyond the scope of this work.
\end{remark}

\subsection{Good stable models as Proj of log canonical algebras}

We prove that good stable models are unique, and coincide with the Proj of the log canonical algebra.

\begin{proposition}\label{prop: Y=proj RX good stable model}
    Let $\phi\colon X\dashrightarrow Y$ be a good stable model over a scheme $T$. Then, $\phibar|_{D_{\Xbar}^n}\colon D_{\Xbar}^n\dashrightarrow D_{\Ybar}^n$ induces a natural injection $R_{D_{\Ybar}^n}\hookrightarrow R_{D_{\Xbar}^n}$ between log canonical algebras, that is equivariant with respect to the gluing involutions and makes
    \begin{equation}\label{diag: commutative restriction}
    \begin{tikzcd}
        R_{D_{\Xbar}^n}\arrow[r,hookleftarrow] & R_{D_{\Ybar}^n}\\
        R_{\Xbar}\arrow[u,"\res_n"] & R_{\Ybar}\arrow[l,"\simeq"]\arrow[u,"\res_n"]
    \end{tikzcd}
    \end{equation}
    commutative.
    In particular, $R_Y\simeq R_X$ and the good stable model $Y\simeq\Proj R_X$ is unique.
\end{proposition}
\begin{proof}
    To simplify the notation, let us assume $T$ to be a point; the general case is analogous.
    
    Let $\Xbar\xleftarrow{p}W\xrightarrow{q}\Ybar$ be a resolution of $\phibar$, and let
    \[
        E=p^*(K_{\Xbar}+D_{\Xbar}+\Delta_{\Xbar})-q^*(K_{\Ybar}+D_{\Ybar}+\Delta_{\Ybar}),
    \]
    which is effective and $q$-exceptional by the properties of log canonical models. Let $Z\subset W$ be the strict transform of $D_{\Xbar}$ along $p$, which coincides with the strict transform of $D_{\Ybar}$ along $q$ by Definition~\ref{def: stable model slc pairs}\eqref{property: birational conductor}. Let $Z^n$ be the normalization of $Z$. We obtain the following commutative diagram:
    \begin{equation}
    \begin{tikzcd}
        & Z^n\arrow[dd]\arrow[ld,"p_D"']\arrow[rd,"q_D"]\\
        D_{\Xbar}^n\arrow[dd]\arrow[rr,dashed, crossing over] & & D_{\Ybar}^n\arrow[dd]\\
        & W\arrow[ld,"p"']\arrow[rd,"q"]\\
        \Xbar\arrow[rr,dashed] & & \Ybar
    \end{tikzcd}
    \end{equation}
    Since $K_{\Xbar}+D_{\Xbar}+\Delta_{\Xbar}$ and $K_{\Ybar}+D_{\Ybar}+\Delta_{\Ybar}$ are $\mathbb{Q}$-Cartier, pulling them back to $W$ and restricting to $Z^n$ yields the same result as taking the different first and then pulling back to $Z^n$. Restricting the divisor equality for $E$ to $Z^n$ thus yields:
    \[
        p_D^*(K_{D_{\Xbar}^n}+\Diff_{D_{\Xbar}^n}(\Delta_{\Xbar})) = q_D^*(K_{D_{\Ybar}^n}+\Diff_{D_{\Ybar}^n}(\Delta_{\Ybar})) + E_Z,
    \]
    where $E_Z := E|_{Z^n}$ remains effective but is no longer $q_D$-exceptional in general. 
    
    Since $\phi$ is a good stable model, the map between conductors is birational, hence $p_D$ and $q_D$ are birational morphisms between normal schemes. Because $E_Z\geq 0$, taking global sections for all sufficiently divisible powers $m\geq 0$ yields a natural inclusion
    \begin{align*}
        \H^0(D_{\Ybar}^n, m(K_{D_{\Ybar}^n}+\Diff_{D_{\Ybar}^n}(\Delta_{\Ybar}))) & \simeq\H^0(Z^n, m q_D^*(K_{D_{\Ybar}^n}+\Diff_{D_{\Ybar}^n}(\Delta_{\Ybar})))\\
        & \subseteq\H^0(Z^n, m p_D^*(K_{D_{\Xbar}^n}+\Diff_{D_{\Xbar}^n}(\Delta_{\Xbar})))\\
        & \simeq\H^0(D_{\Xbar}^n, m(K_{D_{\Xbar}^n}+\Diff_{D_{\Xbar}^n}(\Delta_{\Xbar}))).
    \end{align*}
    Summing over all $m$ and applying~\cite[Lemma 2.19]{Kollar-singularities}, we obtain the desired injection of graded rings $R_{D_{\Ybar}^n}\hookrightarrow R_{D_{\Xbar}^n}$, which makes diagram~\eqref{diag: commutative restriction} commute. The inclusion is $\mu$-equivariant, since the birational map $D_{\Xbar}^n\dashrightarrow D_{\Ybar}^n$ is $\mu$-equivariant, being induced by $\phi\colon X\dashrightarrow Y$.

    The global isomorphism $R_X\simeq R_Y$ follows from diagram~\eqref{diag: commutative restriction} and Lemma~\ref{lem: log canonical ring slc pair}.
    Since $K_Y+\Delta_Y$ is ample by Definition~\ref{def: stable model slc pairs}\eqref{property: ample}, we conclude that $Y\simeq\Proj R_Y\simeq\Proj R_X$.
\end{proof}

Let $(X,\Delta_X)$ be an slc pair, suppose the log canonical ring $R_X$ to be finitely generated, and suppose that $X\dashrightarrow Y:=\Proj R_X$ is birational. Let $\Delta_Y=\phi_*\Delta_X$. If $X$ is normal, then it is well-known (\cite{Reid-canonical} or~\cite[Theorems 1.15, 1.26]{Kollar-singularities}) that $\Proj R_X$ is the log canonical model of $X$.
However, for an slc pair it is not even clear that $Y:=\Proj R_X$ is demi-normal, which indeed fails by Example~\ref{exm: Proj RX not deminormal}. Even then, the relationship between the finite generation of $R_X$ and $R_{\Xbar}$ is somewhat obscure, as the extension $R_X\hookrightarrow R_{\Xbar}$ is not always finite. Indeed, there are cases where $R_{\Xbar}$ is a finitely generated ring but $R_X$ is not, even when $X$ is a normal-crossing surface, see~\cite{Kollar_examples_SNC}.

We prove that $\Proj R_X$ is a good stable model if the map $X\dashrightarrow\Proj R_X$ is regular in codimension 1 and does not contract any component of the conductor of $X$.

\begin{theorem}\label{thm: projRX is stable model}
    Let $(X,\Delta_X)$ be an slc pair proper over a semi-normal scheme $T$, such that each irreducible component $X_i$ of $X$ dominates some irreducible component $T_i$ of $T$ (not necessarily distinct). Suppose that the $\cO_T$-algebra $R_X:=R(X/T,K_X+\Delta_X)$ is finitely generated, and that $K_{X_i}+\Delta_{X_i}$ on $X_i$ is relatively big over $T_i$ for all $i$. Suppose further that the induced map $\phi\colon X\dashrightarrow\Proj R_X$ is regular in codimension 1 and does not contract any component of the conductor $D_X$.
    
    Then, $\phi$ is the unique good stable model of $X$ over $T$, and it induces an identification of log canonical algebras.
    Moreover, if $K_X+\Delta_X$ is relatively nef over $T$, then $\phi$ is a morphism.
    
    In particular, $R_{\Xbar}:=R(\Xbar/T,K_{\Xbar}+D_{\Xbar}+\Delta_{\Xbar})$ is also finitely generated and the induced map $\Proj R_{\Xbar}\dashrightarrow\Proj R_X$ is a morphism that coincides with the normalization of $\Proj R_X$.
\end{theorem}

First, we prove a proposition in a more general setting.

\begin{proposition}\label{prop: properties of Proj deminormal}
    Let $X$ be a demi-normal scheme, with a proper morphism $f\colon X\rightarrow T$ to a semi-normal scheme $T$, such that each irreducible component $X_i$ of $X$ dominates some irreducible component $T_i$ of $T$ (not necessarily distinct). Let $L$ be an effective Cartier divisor on $X$.
    Suppose that $R(X/T,L)$ is (finitely) generated by $f_*\cO_X(L)$ as an $\cO_T$-algebra, and that $L_i:=L|_{X_i}$ on $X_i$ is relatively big over $T_i$ for all $i$.
    Let $\phi\colon X\dashrightarrow\PP_T^N$ be the map given by $|L|_T$. Let $Y$ be the closure of the image and $|H|_T$ the hyperplane class on $Y$. Suppose that $\Ex(\phi)$ does not contain any component of the conductor $D_X$. Then:
    \begin{enumerate}
        \item $\phi$ is regular in codimension 1 and $\phi$ is birational
        \item $Y$ is demi-normal.
        \item $\phi_*^{-1}D_Y=D_X$ and $\phi_*D_X=D_Y$.
        \item $Y\setmin\phi(X\setmin\Bs|L|_T)$ has codimension $\geq2$.
        \item Every divisor in $\Bs|L|_T$ is contracted by $\phi$.
        \item $\phi_*|L|_T=|H|_T$.
        \item If $L$ is nef over $T$, then $|L|_T$ is base-point free over $T$.
    \end{enumerate}
\end{proposition}
\begin{proof}
    To simplify the notation, we prove the statement in the case where $T$ is a point. The argument in the general case is analogous, and can be reduced to the case where $T$ is affine.
    
    By assumption, $\phi$ is a morphism around the generic points of $D_X$, and can be extended to a morphism at the other codimension 1 points as they are regular. Secondly, note that for $m\gg0$ the sections of $H^0(X,mL)$ separate the irreducible components of $X$, and grow as $c'm^{\dim X}$ for some constant $c'>0$. This follows immediately from the fact that, for every irreducible component $X_i$ of $X$, the dimension of the target of the restriction homomorphism
    \[
    \begin{tikzcd}
        \H^0(X_i,mL_i)\arrow[r] & \oplus_{j\not=i}\H^0(X_i\cap X_j,mL|_{X_i\cap X_j})
    \end{tikzcd}
    \]
    grows at most as $m^{\dim X-1}$. The birationality of $\phi|_{X_i}$ follows from~\cite[Lemma 3.23]{Kollar_Mori_1998}.

    Now, let $X\xleftarrow{p}W''\xrightarrow{q}Y$ be the closure of the graph of $\phi$, with the two projections. Since $\phi$ is defined on an open $U$ whose complement has codimension at least 2, $p\colon W''\rightarrow X$ is an isomorphism over $U$ and $p^{-1}(U)\subset W''$ is at most nodal in codimension 1. However, it is possible for $W''$ to have singular point of codimension 1 that get contracted to a subscheme of $X\setminus U$. To fix this, we first take a semi-normalization of $W''$ and then blowup the singular codimension 1 points outside $p^{-1}(U)$. The resulting scheme $W'$ is now nodal in codimension 1, thus we can take its demi-normal modification $W\rightarrow W'$ (Definition~\ref{def: deminormal modification}). We still use $p$ and $q$ to denote the maps from $W$ to $X$ and $Y$, respectively. Then, $p\colon W\rightarrow X$ is a birational morphism between demi-normal varieties, induces an isomorphism $p^{-1}(U)\rightarrow U$, and all singular points of codimension 1 in $X$ (resp. $W$) are contained in $U$ (resp. $p^{-1}(U)$). It follows that $p_*\cO_W\simeq\cO_X$. Therefore, if we set $M:=p^*L$, the pair $(W,M)$ satisfies the same assumptions as in the statement, and $\Bs|M|=p^{-1}\Bs|L|$. It is then enough to prove the statement assuming that $\phi$ is a surjective proper morphism, and
    \begin{equation}\label{eq: relation linear systems}
        |L|=\phi^*|H|+F
    \end{equation}
    where $F$ is a Cartier divisor equal to the base locus of $|L|$. In particular, $\H^0(Y,mH)\simeq\H^0(X,mL)$ for all $m\geq0$.

    Now, let $X\xrightarrow{\phi'}Y'\xrightarrow{\psi}Y$ be any factorization of $\phi$ with $\phi'$ a birational morphism and $\psi$ finite. Then, $m\psi^*H=\psi^*(mH)$ is very ample for $m\gg0$, and
    \[
        \H^0(Y,mH)\subseteq\H^0(Y',m\psi^*H)\subseteq\H^0(X,mL),
    \]
    as the morphisms are dominant and satisfy an equality between linear systems as in~\eqref{eq: relation linear systems}. Since the composite of inclusions is an equality, we then get $Y'\simeq Y$. It follows that any such factorization is the trivial one, in particular $\phi_*\cO_X=\cO_Y$ and the fibers are connected.
    
    We are ready to show that $Y$ is demi-normal. Let $\eta\in Y$ be a codimension 1 point. By surjectivity, birationality, and connectedness of the fibers, there is a unique point $\xi\in X$ mapping to $\eta$, it has codimension 1 and same residue field. In particular, $\phi$ is an isomorphism around $\eta$, which then has the same local behavior as $\xi$. This implies that $Y$ is nodal in codimension 1. Since $\phi$ does not contract any component of the conductor, by taking the demi-normal modification $\widetilde{Y}\rightarrow Y$, we obtain a factorization $X\rightarrow\widetilde{Y}\rightarrow Y$ of $\phi$ in birational morphisms, with the second map being finite (Lemma~\ref{lem: universal property deminormal modification}). By what we have proved above, it follows that $\widetilde{Y}=Y$ hence $Y$ is demi-normal.
    
    Note that we have also proved that all Weil divisors $B$ in $X$ are either contracted or mapped to a Weil divisor in $Y$, and that all irreducible components of $D_X$ are either contracted or mapped to an irreducible component of $D_Y$.
    Since the components of the conductor are not contracted by assumption, we get that $\phi_*^{-1}D_Y=D_X$ and $\phi_*D_X=D_Y$.
    
    The same argument as in~\cite[Proposition 1.16]{Kollar-singularities} shows that if a Weil divisor $B$ is not contracted then $B\not\subset F=\Bs|L|$. Since no component of $D_X$ is contained in the base locus, it follows that $\phi(\Bs|L|)$ has codimension at least 2 in $Y$, and $\phi_*|L|=|H|$.

    The fact that $|L|$ is base-point free when $L$ is nef follows from the next lemma.
\end{proof}

\begin{lemma}
    Let $\phi\colon X\rightarrow Y$ be a proper birational morphism between demi-normal schemes, and $F$ an effective $\QQ$-Cartier $\QQ$-divisor that is $\phi$-nef and contracted by $\phi$. Then, $F=0$.
\end{lemma}
\begin{proof}
    The result is standard when $X$ and $Y$ are normal, and it follows from~\cite[Lemma 3.39]{Kollar_Mori_1998}. Let $\phibar\colon \Xbar\rightarrow\Ybar$ be the induced morphism between the normalizations, and note that $\pi_X^*F$ is again effective, $\phibar$-nef, and contracted by $\phibar$.
    Therefore, $\pi_X^*F=0$, thus $F=0$.
\end{proof}

\begin{proof}[Proof of Theorem~\ref{thm: projRX is stable model}]
    As for Proposition~\ref{prop: properties of Proj deminormal}, we provide a proof in the case where $T$ is a point. The general case is analogous.
    
    Since $R_X$ is finitely generated, there is an integer $\ell>0$ such that $\ell(K_X+\Delta_X)$ is Cartier and the $\ell$-th Veronese subalgebra of $R_X$ is generated by $\H^0(X,\ell(K_X+\Delta_X))$. Then, $L:=\ell(K_X+\Delta)$ satisfies the assumptions of Proposition~\ref{prop: properties of Proj deminormal}.
    Therefore, there is a natural birational map $\phi\colon X\dashrightarrow Y$ regular in codimension 1 and with demi-normal target, together with a very ample Cartier divisor $H$ on $Y$ satisfying $H\sim\phi_*L$, and $\H^0(X,mL)\simeq\H^0(Y,mH)$ via $\phi$, for all $m\geq0$. Moreover, $\Delta_Y:=\phi_*\Delta_X$ is a Weil divisor since $\phi^{-1}_*D_Y=D_X$.
    Then,
    \[
        \ell(K_Y+\Delta_Y)\sim\phi_*\ell(K_X+\Delta_X)=\phi_*L\sim H,
    \]
    hence $K_Y+\Delta_Y$ is $\QQ$-Cartier and ample. Since $\phi$ does not contract the conductor by assumption, by Lemma~\ref{lem: redundant def stable model} we are only left with showing that $\phibar$ is a log canonical model.

    For this, we can reduce to the case where $\phi$ is a morphism, by taking the demi-normal modification of a partial resolution of the closure of the graph of $\phi$ as done in the proof of Proposition~\ref{prop: properties of Proj deminormal}. Recall that we have
    \[
        \phi^*(\ell(K_Y+\Delta_Y))+F=\ell(K_X+\Delta_X)
    \]
    with $F$ effective and $\phi$-exceptional. Pulling back to the normalization, we get
    \[
        \phibar^*(\ell(K_{\Ybar}+D_{\Ybar}+\Delta_{\Ybar}))+\pi_X^*F=\ell(K_{\Xbar}+D_{\Xbar}+\Delta_{\Xbar}).
    \]
    Then, $\H^0(\Ybar,m(K_{\Ybar}+D_{\Ybar}+\Delta_{\Ybar}))=\H^0(\Xbar,m(K_{\Xbar}+D_{\Xbar}+\Delta_{\Xbar}))$ for all $m\geq0$ sufficiently divisible. It follows that $R_{\Xbar}=R(\Xbar,K_{\Xbar}+D_{\Xbar}+\Delta_{\Xbar})$ is finitely generated and $\Proj R_{\Xbar}=\Ybar$ is indeed the log canonical model. Moreover, by Proposition~\ref{prop: properties of Proj deminormal} we know that $\phibar_*D_{\Xbar}=D_{\Ybar}$, while $\phi_*\Delta_{\Xbar}=\Delta_{\Ybar}$ holds by definition of $\Delta_{Y}$.
\end{proof}

\begin{example}\label{exm: Proj RX not deminormal}
    We provide an example showing that $\Proj R_X$ can fail to be demi-normal if the exceptional locus of the natural map $X\dashrightarrow \Proj R_X$ contains a component of $D_X$.

    Let $S$ be a minimal smooth projective surface containing a smooth elliptic curve $E$. Let $T$ be the log canonical model of $(S,E)$, which is obtained by contracting $E$ to a simple elliptic singularity $p\in T$. Let $X$ be the union of two copies of $S$ glued along $E$, and let $R_X:=R(X,K_X)$. Note that the restriction map
    \[
    \begin{tikzcd}
        \res\colon R_S:=R(S,K_S+E)\arrow[r] & R(E,K_E)=:R_E\simeq\CC[t]
    \end{tikzcd}
    \]
    is surjective by~\cite[Proposition 4.45]{Kollar_Mori_1998} or Kawamata-Viehweg vanishing. Moreover, for every $m\geq0$ there exists an exact sequence
    \[
    \begin{tikzcd}[column sep=small]
        0\arrow[r] & \H^0(X,mK_X)\arrow[r] & \H^0(S,m(K_S+E))\oplus \H^0(S,m(K_S+E))\arrow[r,"\phi"] & \H^0(E,mK_E)\arrow[r] & 0
    \end{tikzcd}
    \]
    where $\phi(s_1,s_2)=\res(s_1)-(-1)^m\res(s_2)$. In particular, $R_X\simeq R_S\times_{\CC[t]}R_S$.
    From this and the surjectivity of $\res$, it follows that $R_X$ is finitely generated and $\Proj R_X$ is the union of two copies of $T$ glued along the elliptic singularity $p$, hence not demi-normal.
\end{example}

An important question is whether a good stable model of a semi-dlt pair induces another good stable model for the conductor.
This fails in general, but there are important situations where the answer is positive.

\begin{lemma}\label{lem: if birational conductor then stable model}
    Let $\phi\colon (X,\Delta_X)\rightarrow(Y,\Delta_Y)$ be a good stable model over $T$ of a semi-dlt pair $(X,\Delta_X)$. Suppose further that $\phi$ is a crepant birational morphism, that is, $K_X+\Delta_X=\phi^*(K_Y+\Delta_Y)$, and that $\Ex(\phi)$ does not contain any stratum of the conductor $D_X$. Let $\phi_D\colon D_{\Xbar}\rightarrow\widetilde{D}_{\Ybar}$ be the morphism induced by taking the Stein factorization of $D_{\Xbar}\rightarrow D_{\Ybar}$. Then, $\widetilde{D}_{\Ybar}$ is demi-normal, $\phi_{D_{\Xbar}}$ is a crepant good stable model over $T$, and $\Ex(\phi_D)$ does not contain any stratum of the conductor. Moreover, $\widetilde{D}_{\Ybar}\simeq\Proj R_{D_{\Xbar}}$.
\end{lemma}

This is a special case of the following more general lemma.

\begin{lemma}\label{lem: birational coeff 1 boundary then stable model}
    Let $\phi\colon (X,\Delta_X)\rightarrow(Y,\Delta_Y)$ be a log canonical model (so $X$ is normal) over a scheme $T$ that is also a crepant birational morphism, meaning $K_X+\Delta_X=\phi^*(K_Y+\Delta_Y)$, and assume $(X,\Delta_X)$ is a dlt pair. Let $S_X\leq\lfloor\Delta_X\rfloor$ be a reduced subdivisor of the coefficient 1 part of the boundary, such that $\phi|_{S_X}\colon S_X\rightarrow Y$ is birational onto its image $S_Y\leq\lfloor\Delta_Y\rfloor$ and does not contract any stratum of $S_X$. Let $\phi_1\colon S_X\rightarrow\widetilde{S}_{Y}$ be the morphism induced by taking the Stein factorization of $S_X\rightarrow S_Y$. Then, $(S_X,\Diff_{S_X}(\Delta_X-S_X))$ is semi-dlt, $\widetilde{S}_{Y}$ is demi-normal, $\phi_1$ is a crepant good stable model over $T$, and $\Ex(\phi_1)$ does not contain any stratum of $S_X$. Moreover, $\widetilde{S}_{Y}\simeq\Proj R_{S_X}$.
\end{lemma}
\begin{proof}
    By the standard (semi-)dlt adjunction in Lemma~\ref{lem: semidlt adjunction} and~\cite[Corollary 4.5]{fujino2016vanishing}, the restriction to the subdivisor $S_X$ yields a well-defined semi-dlt pair $(S_X, \Diff_{S_X}(\Delta_X-S_X))$. In particular, $S_X$ is a demi-normal scheme.
    
    Taking the Stein factorization of the proper morphism $S_X\rightarrow S_Y$ yields a morphism $\phi_1\colon  S_X \rightarrow\widetilde{S}_Y$ with connected fibers. Then, $\widetilde{S}_Y$ is demi-normal because the domain $S_X$ is demi-normal, the map does not contract any component of the conductor of $S_X$, and $S_Y$ is nodal in codimension 1, see Example~\ref{exm: deminormal modifications}\eqref{exm: deminormal stein factorization}. Restricting the crepant identity $K_X+S_X+(\Delta_X-S_X)=\phi^*(K_Y+\Delta_Y)$ to $S_X$ via adjunction (or \cite[Proposition 4.6]{Kollar-singularities}) shows that $\phi_1$ is a crepant morphism of pairs.
    Since $\phi_1$ is a crepant morphism with connected fibers from a semi-dlt pair, the target pair $(\widetilde{S}_Y, \Delta_{\widetilde{S}_Y})$ is automatically slc. Furthermore, since $K_Y+\Delta_Y$ is ample by the definition of a log canonical model and $\widetilde{S}_Y$ is finite over $Y$, $(\widetilde{S}_Y,\Delta_{\widetilde{S}_Y})$ is a stable slc pair.
    Since $\phi_1$ is a crepant birational morphism from a semi-dlt pair to a target polarized by an ample log canonical class, it satisfies every condition of a good stable model over $T$. The isomorphism $\widetilde{S}_{Y}\simeq\Proj R_{S_X}$ follows from Proposition~\ref{prop: Y=proj RX good stable model}.
\end{proof}

\section{Proofs of the main results}\label{sec: proof main theorems}

In this section, we prove Theorems~\ref{thm: general version} and~\ref{thm: equivariant sklt}, starting from the latter.

\subsection{Proof of Theorem~\ref{thm: equivariant sklt}}\label{subsec: proof equivariant sklt}

We start by recalling the setting. Let $(X,\Delta_X)$ be a semi-klt pair, $T$ a semi-normal scheme, and $f\colon X\rightarrow T$ a surjective projective morphism. Let $T^0\subset T$ be an open dense subscheme, and set $X^0=T^0\times_{T}X$, $\Delta_{X^0}=\Delta_X|_{X^0}$. We may also assume that $X$ is not normal, hence $D_X\not=\emptyset$. We assume that $X^0$ intersects all log centers of $(X,\Delta_X)$ lying on $D_X$; in particular, $X^0$ intersects all strata of the conductor $D_X$ and all irreducible components of $X$.
Suppose that $(X^0,\Delta_{X^0})$ admits a stable model
\[
\begin{tikzcd}
    \phi^{0,c}\colon (X^0,\Delta_{X^0})\arrow[r,dashed] & (X^{0,c},\Delta_{X^{0,c}})
\end{tikzcd}
\]
whose exceptional locus does not contain any stratum of $D_{X^0}$, and that the normalization pair $(\Xbar,D_{\Xbar}+\Delta_{\Xbar})$ admits a log canonical model $\phibar^c\colon \Xbar\rightarrow(\Xbar^c,D_{\Xbar^c}+\Delta_{\Xbar^c})$. Let $G\subset\Bir(\Xbar/T,D_{\Xbar}+\Delta_{\Xbar})$ be a finite subgroup. Let us set $R_X:=R(X/T,K_X+\Delta_X)$, and similarly for $R_{\Xbar}$.
Our goal is then to show that the extension of $\cO_T$-algebras
\[
\begin{tikzcd}
    R_X\cap R_{\Xbar}^G\arrow[r,hookrightarrow] & R_{\Xbar}
\end{tikzcd}
\]
is finite, with the convention of Definition~\ref{def: invariants under Bir}. In particular, we will tacitly restrict to suitable $\ell$-Veronese subalgebras with $\ell$ even; see Remark~\ref{rmk: different invariants mu} for a discussion on the parity of $\ell$.

We prove this by induction on the dimension of $X$. The base case is for 0-dimensional pairs, thus trivially true.
For the inductive step, we need the following lemmas, that highlights the nice properties specific to semi-klt pairs.

\begin{lemma}\label{lem: conductor for sklt induction}
    With the above assumptions, there exists a dlt pair $(\Xbar^m,D_{\Xbar^m}+\Delta_{\Xbar^m})$
    fitting in a commutative diagram of rational maps over $T$
    \[
    \begin{tikzcd}
        D_{\Xbar}^n\arrow[d]\arrow[r,dashed,"\phibar_{D^n}^m"] & D_{\Xbar^m}^n\arrow[d]\arrow[r,"\phibar_{D^n}"] & D_{\Xbar^c}^n\arrow[d]\\
        D_{\Xbar}\arrow[r,dashed,"\phibar_D^m"]\arrow[d,hookrightarrow] & D_{\Xbar^m}\arrow[r,"\phibar_D"]\arrow[d,hookrightarrow] & D_{\Xbar^c}\arrow[d,hookrightarrow]\\
        \Xbar\arrow[d,"\pi_X"]\arrow[r,dashed,"\phibar^m"] & \Xbar^m\arrow[r,"\phibar"] & \Xbar^c\\
        X
    \end{tikzcd}
    \]
    where solid arrows are morphisms, such that the following hold.
    \begin{enumerate}[itemsep=1em]
        \item\label{point 1} $\phibar^m$ is the minimal model of $(\Xbar,D_{\Xbar}+\Delta_{\Xbar})$, and $\phibar^c=\phibar\circ\phibar^m$ the log canonical model, with $\phibar_*^m D_{\Xbar}=D_{\Xbar^m}$ and $\phibar_* D_{\Xbar^m}=D_{\Xbar^c}$.
        \item\label{point 2} $\phibar_D$, $\phibar_D^m$ and $\phibar_D^c:=\phibar_D\circ\phibar_D^m$ are birational, and their exceptional loci do not contain any stratum of their conductors. The same holds for their normalizations $\phibar_{D^n}$, $\phibar_{D^n}^m$ and $\phibar_{D^n}^c:=\phibar_{D^n}\circ\phibar_{D^n}^m$.
        \item\label{point 3} $\lfloor\Delta_{\Xbar^m}\rfloor=0$ and $\lfloor\Diff_{D_{\Xbar^m}}(\Delta_{\Xbar^m})\rfloor=0$.
        \item\label{point 4} $\phibar_D$ is crepant and a good stable model of the semi-klt pair $(D_{\Xbar^m},\Diff_{D_{\Xbar^m}}(\Delta_{\Xbar^m}))$. The same holds for $\phibar_{D^n}$. In particular, $D_{\Xbar^c}$ is demi-normal, and $(D_{\Xbar^c},\Diff_{D_{\Xbar^c}}(\Delta_{\Xbar^c}))$ is a stable slc pair.
        \item\label{point 5} Let $R_{D_{\Xbar}^n}:=R(D_{\Xbar}^n/T,K_{D_{\Xbar}^n}+\Diff_{D_{\Xbar}^n}(\Delta_{\Xbar}))$, and similarly for $R_{D_{\Xbar^m}^n}$ and $R_{D_{\Xbar^c}^n}$. Then,
        \[
            R_{D_{\Xbar}^n}\supset R_{D_{\Xbar^m}^n}\simeq R_{D_{\Xbar^c}^n}.
        \]
    \end{enumerate}
\end{lemma}
\begin{proof}
    Since $\Xbar$ admits a log canonical model, eventually after taking a small $\QQ$-factorial modification, the MMP with scaling of an ample divisor terminates with a $\QQ$-factorial good minimal model $\phibar^m\colon \Xbar\dashrightarrow\Xbar^m$, mapping to a log canonical model $\Xbar^c$ (see for instance~\cite{Fujino2011_MMPtermination} and~\cite[\S2]{Hacon_Xu_dslt}). This yields~\eqref{point 1} and the diagram.

    Since $X^0$ intersects all strata of the conductor of $X$, to prove~\eqref{point 2}, it is enough to show that $\Ex(\phibar_D^c)$ does not contain any stratum of the conductor when restricted to $T^0$. On the other hand, $\Xbar^{0,c}=\Xbar^c\times_{T}T^0$ is the normalization of $X^{0,c}$, thus the claim follows from the assumption that the exceptional locus of $\phi^{0,c}$ does not contain any stratum of the conductor $D_X\subset X$.

    The first equality of~\eqref{point 3} follows immediately from $\phibar_{D*}^m\Delta_{\Xbar}=\Delta_{\Xbar^m}$. The second equality is proved in Lemma~\ref{lem: semidlt adjunction}.

    To prove~\eqref{point 4}, it is enough to show that $\phibar_D$ is a good stable model, as $\phibar_D$ is crepant by either~\cite[Proposition 4.6]{Kollar-singularities} or Lemma~\ref{lem: birational coeff 1 boundary then stable model}. By Lemma~\ref{lem: if birational conductor then stable model}, $D_{\Xbar^m}$ has finitely generated log canonical algebra $R_{D_{\Xbar^m}}$, whose Proj is a stable model $Z$ that is finite over $D_{\Xbar^c}$. By Lemma~\ref{lem: surjectivity restriction minimal model} below, the restriction map $R_{\Xbar^m}^{(\ell)}\twoheadrightarrow R_{D_{\Xbar^m}}^{(\ell)}$ is surjective for $\ell$ sufficiently divisible. It follows that $\Proj R_{D_{\Xbar^m}}\rightarrow\Proj R_{\Xbar^m}=\Xbar^c$ is a closed immersion. In particular, $Z=\Proj R_{D_{\Xbar^m}}\rightarrow D_{\Xbar^c}$ is an isomorphism, being a map of reduced subschemes with the same underlying set.

    The isomorphism in~\eqref{point 5} follows from $\phibar_{D^n}$ being a log canonical model, while the inclusion on the left follows from Proposition~\ref{prop: Y=proj RX good stable model}.
\end{proof}

\begin{lemma}\label{lem: surjectivity restriction minimal model}
    Let $(Y,S+B)$ be a $\QQ$-factorial dlt pair with $\lfloor S+B\rfloor=S$. Let $f\colon Y\rightarrow T$ be a proper morphism to a scheme $T$ and suppose that $K_Y+S+B$ is $f$-nef and $f$-big. Then, for $\ell$ sufficiently divisible, $\ell(K_Y+S+B)$ is Cartier and the restriction map
    \[
    \begin{tikzcd}
        f_*\cO_Y(\ell(K_X+S+B))\arrow[r,twoheadrightarrow] & f_*\cO_{S}(\ell(K_S+\Diff_S(B)))
    \end{tikzcd}
    \]
    is surjective.
\end{lemma}
\begin{proof}
    Let $L=K_Y+S+B$, which is $f$-nef and $f$-big by assumption. Let $\ell$ be sufficiently divisible, so that $\ell L$ is Cartier and $\ell B$ is a $\ZZ$-divisor. It is enough to show that $R^1f_*\cO_Y(\ell L-S)$ is trivial. Since the normalization of $T$ is affine over $T$, we can prove this under the further assumption of $T$ being normal. Then, this is a standard consequence of~\cite[Theorem 10.37]{Kollar-singularities}. Indeed, write
    \[
        \ell L-S=K_Y+(\ell-1)L+B.
    \]
    Since $(Y,B)$ is klt, we can apply Kawamata-Viehweg vanishing~\cite[Theorem 10.37]{Kollar-singularities}.
\end{proof}

\begin{remark}
    At first glance, the statement of Lemma~\ref{lem: conductor for sklt induction}\eqref{point 4} seems surprising in general, as we are not explicitly asking $\phi^{0,c}\colon X^0\dashrightarrow X^{0,c}$ to induce a bijection between the set of \emph{all} strata of the respective conductors, and a birational map between \emph{all} such strata. Indeed, suppose $X^0$ to be SNC and nodal, but $X^{0,c}$ has a non-empty SNC locus of higher codimension. For a local picture, $X^{0,c}$ could be the zero locus of $x_1x_2x_3\in\CC[x_1,x_2,x_3]$ in $\AA^3$, and $X^{0}$ its blowup in $0$. Then, the conductor in $\Xbar^0$ would be the normalization of the conductor in $\Xbar^{0,c}$; in particular, the induced map is not Stein. However, this cannot happen under our assumptions, and the birationality of the strata is hidden in the semi-klt condition. Indeed, the discrepancy at the generic point of any stratum of the conductor in $X^{0,c}$ is $-1$, thus contradicting the fact that $\lfloor\Delta_{X^0}\rfloor=0$.
\end{remark}

We are ready to prove the inductive step for Theorem~\ref{thm: equivariant sklt}; we use the notation of Lemma~\ref{lem: conductor for sklt induction}.

\begin{proof}[Proof of Theorem~\ref{thm: equivariant sklt}]
First, note that there is an honest $G$-action on the canonical model $(\Xbar^c,D_{\Xbar^c}+\Delta_{\Xbar^c})$, as $\Xbar^c=\Proj R_{\Xbar}$ and $G$ acts on a Veronese subalgebra of $R_{\Xbar}$ by Lemma~\ref{lem: Bir acts on log canonical}.
This action restricts to a $G$-action on $(D_{\Xbar^c}^n,\Diff_{D_{\Xbar^c}^n}(\Delta_{\Xbar^c}))$. Moreover, there is a canonical involution $\mu$ on this pair, defined as follows. Recall that $(\Xbar^{0,c},D_{\Xbar^{0,c}}+\Delta_{\Xbar^{0,c}})$ is the normalization of the stable model $(X^{0,c},\Delta_{X^{0,c}})$ of $X$ over $T^0$. Therefore, there is a gluing involution $\mu^0$ on $(D_{\Xbar^c}^n,\Diff_{D_{\Xbar^c}^n}(\Delta_{\Xbar^c}))|_{T^0}$, see Remark~\ref{rmk: normalization pair}. Since all log centers on $D_{\Xbar^c}$ intersect $\Xbar^{0,c}$ by assumption (and~\cite[Proposition 3.51]{Kollar_Mori_1998}), \cite[Theorem 11.40]{Kollar-moduli} implies that $\mu^0$ extends to an involution $\mu$ on the whole pair $(D_{\Xbar^c},\Diff_{D_{\Xbar^c}}(\Delta_{\Xbar^c}))$.
Let $H=\langle G,\mu\rangle$ be the subgroup of $\Aut(D_{\Xbar^c}^n/T,\Diff_{D_{\Xbar^c}^n}(\Delta_{\Xbar^c}))$ generated by $G$ and $\mu$, which is finite as the automorphism group of a stable pair is finite.

Since $(D_{\Xbar^c}^n,\Diff_{D_{\Xbar^c}^n}(\Delta_{\Xbar^c}))$ is the log canonical model of $(D_{\Xbar^m}^n,\Diff_{D_{\Xbar^m}^n}(\Delta_{\Xbar^m}))$ and the morphism $D_{\Xbar^m}^n\rightarrow D_{\Xbar^c}^n$ is crepant by Lemma~\ref{lem: conductor for sklt induction}\eqref{point 4}, we get a canonical inclusion
\[
    H\subset\Bir(D_{\Xbar^m}^n/T,\Diff_{D_{\Xbar^m}^n}(\Delta_{\Xbar^m}))
\]
by Remark~\ref{rmk: Bir is a crepant invariant}.
By the inductive assumption, the inclusions
\begin{equation}\label{eq: finite inclusion conductor}
\begin{tikzcd}
    \widetilde{R}_D^H:=R_{D_{\Xbar^m}^n}^H\cap R_{D_{\Xbar^m}}\arrow[r,hookrightarrow] & R_{D_{\Xbar^m}} \arrow[r,hookrightarrow] & R_{D_{\Xbar^m}^n}
\end{tikzcd}
\end{equation}
are finite extensions. Now, recall that $R_{\Xbar}\simeq R_{\Xbar^m}$, and by the inclusions in Lemma~\ref{lem: conductor for sklt induction}\eqref{point 5} we have a commutative diagram
\begin{equation}\label{diag: restriction}
\begin{tikzcd}
    R_{D_{\Xbar}^n}\arrow[r,hookleftarrow] & R_{D_{\Xbar^m}^n}\arrow[r,phantom,"\simeq"] & R_{D_{\Xbar^c}^n}\\
    R_{\Xbar}\arrow[r,phantom,"\simeq"]\arrow[u,"\res_n"] & R_{\Xbar^m}\arrow[u,"\res_n"]\arrow[r,phantom,"\simeq"] & R_{\Xbar^c}\arrow[u,"\res_n"]
\end{tikzcd}
\end{equation}
showing that
\[
    R_X=
    \res_n^{-1}(R_{D_{\Xbar^m}^n}^{\mu})\supset\res_n^{-1}(R_{D_{\Xbar^m}^n}^H)=\res^{-1}(\widetilde{R}_D^H),
\]
where we have used Lemma~\ref{lem: log canonical ring slc pair} and Remark~\ref{rmk: different invariants mu}. Since the inclusion in~\eqref{eq: finite inclusion conductor} is a finite extension, it is in particular integral. By Lemma~\ref{lem: surjectivity restriction minimal model}, the restriction map between $\ell$-th Veronese subalgebras
\[
\begin{tikzcd}
    \res\colon R_{\Xbar^m}^{(\ell)}\arrow[r] & R_{D_{\Xbar^m}}^{(\ell)}
\end{tikzcd}
\]
is surjective for $\ell$ sufficiently divisible. In particular, the pullback of~\eqref{eq: finite inclusion conductor}
\begin{equation}\label{eq: extension}
\begin{tikzcd}
    \res^{-1}(\widetilde{R}_D^H)\arrow[r,hookrightarrow] & R_{\Xbar^m}
\end{tikzcd}
\end{equation}
induces an integral extension between appropriate Veronese subalgebras. Since any graded algebra is integral over any of its Veronese subalgebras, the extension in~\eqref{eq: extension} is also an integral extension, hence finite, as $R_{\Xbar^m}$ is of finite type over $T$.

Since $\res^{-1}(\widetilde{R}_D^H)\subset R_X\subset R_{\Xbar}\simeq R_{\Xbar^m}$, the extension $R_X\subset R_{\Xbar}$ is also finite; in particular, $(X,\Delta_X)$ has finitely generated log canonical algebra, thus admitting a canonical model by Theorem~\ref{thm: projRX is stable model}.

We also have inclusions $R_X\cap R_{\Xbar}^G\subset\res^{-1}(\widetilde{R}_D^H)\subset R_{\Xbar}$, hence to prove the theorem it is enough to show that the first inclusion is integral, for which we can pass to Veronese subalgebras as needed. Note that $R_X\cap R_{\Xbar}^G\subset\res^{-1}(\widetilde{R}_D^H)$ is not an equality: $\alpha \in R_X^{(\ell)}$ maps to $\widetilde{R}_D^H$ if its image $\res(\alpha)$ in $R_{D_{\Xbar^m}^n}$ is invariant under $G$, but $\alpha$ could still fail to be $G$-invariant as a section of $R_{\Xbar}$, as of course the restriction map $\res_n$ is not injective.

Since $G$ is linearly reductive, the map
\[
\begin{tikzcd}
    \mathrm{res}^G\colon (R_{\Xbar}^{(\ell)})^G\simeq (R_{\Xbar^m}^{(\ell)})^G \arrow[r] & (R_{D_{\Xbar^m}}^{(\ell)})^G
\end{tikzcd}
\]
is again surjective for $\ell$ sufficiently divisible. Let $I$ be the kernel ideal of the restriction morphism $\res$, hence $I^G=I\cap R_{\Xbar^m}^G$ is the kernel of $\res^G$.

Let $\alpha\in\res^{-1}(\widetilde{R}_D^H)$ have even, sufficiently divisible degree, and
\[
    \bar{\alpha} = \mathrm{res}(\alpha) \in R_{D_{\Xbar^m}}\cap R_{D_{\Xbar^m}^n}^H \subset R_{D_{\Xbar^m}}^G
\]
its image. By surjectivity of $\res^G$, there exists $\beta\in R_{\Xbar}^G$ s.t. $\res(\beta) = \bar{\alpha}$.
In particular, $\mathrm{res}(\beta) \in R_{D_{\Xbar^m}}\cap R_{D_{\Xbar^m}^n}^\mu$, hence $\beta \in R_X\cap R_{\Xbar}^G$. By construction, $\alpha-\beta\in I$.
Therefore, to conclude it is enough to show that every $\alpha \in I$ of sufficiently divisible degree is integral over $R_X\cap R_{\Xbar}^G$.

Note that $I \subset R_X$, as the zero-function on $R_{D_{\Xbar}^n}$ is trivially $\mu$-invariant.
Consider the monic polynomial $P(X) := \prod_{g \in G}(X - g \cdot \alpha)$, hence $P(\alpha) = 0$ (look at $g = \text{identity}$).
The polynomial is $G$-invariant by construction, with coefficients being (up to sign) symmetric functions in $g \cdot \alpha$ for $g \in G$ varying. In particular, the coefficients are in $I^G \subset R_X\cap R_{\Xbar}^G$, proving the result.

The fact that $\Proj R_X$ is the good stable model of $(X,\Delta_X)$ over $T$ is guaranteed by Theorem~\ref{thm: projRX is stable model}. Note that there could be vertical irreducible components of $X$ over $T$, but they have image in $T^0$ by assumption, where we already know $\Proj R_X$ to be the stable model by assumption.
\end{proof}

\subsection{Proof of Theorem~\ref{thm: general version}}\label{subsec: proof main theorem}

The idea of the proof of Theorem~\ref{thm: general version} is to perturb the semi-dlt pair $(X,\Delta_X)$ by adding $\epsilon A_X$ for a relatively ample divisor $A_X$ over $T$, with $\epsilon>0$ sufficiently small. The resulting pair is equivalent to a semi-klt pair, so that we can apply Theorem~\ref{thm: equivariant sklt} to obtain a stable model $\phi^c_+\colon X\dashrightarrow X^c_+$ that does not depend on $\epsilon$ for $0<\epsilon\ll1$. One difficulty is that there is a jump behavior at $\epsilon=0$, where the canonical model can be different, if it exists. To deduce the finite generation of $R_X=R(X/T,K_X+\Delta_X)$, we show that the $K_{X^c_+}+(\phi^c_+)_*\Delta_X$ is a semi-ample $\QQ$-Cartier divisor over $T$.

Let us first reduce the statement to the case where $T$ is projective. Recall that any scheme is assumed to be reduced, separated, and essentially of finite type over $\CC$.

\begin{lemma}\label{lem: reduction to projective case}
    Suppose that Theorem~\ref{thm: general version} holds true when $T$ is further assumed to be projective. Then it holds for any semi-normal scheme $T$.
\end{lemma}
\begin{proof}
    First of all, note that the statement is local on the base, thus we may assume $T$ to be affine. Moreover, by spreading out we can assume that the data
    \[
    \begin{tikzcd}
        X\arrow[d,"f"]\arrow[r,hookleftarrow] & X^0\arrow[d,"f^0"]\arrow[r,dashed,"\phi^{0,c}"] & X^{0,c}\arrow[ld]\\
        T\arrow[r,hookleftarrow] & T^0
    \end{tikzcd}
    \]
    is obtained by localization from some other commutative diagram with cartesian square
    \[
    \begin{tikzcd}
        Y\arrow[d,"g"]\arrow[r,hookleftarrow] & Y^0\arrow[d,"g^0"]\arrow[r,dashed,"\psi^{0,c}"] & Y^{0,c}\arrow[ld]\\
        S\arrow[r,hookleftarrow] & S^0
    \end{tikzcd}
    \]
    where $S$ is affine and of finite type, $g$ is projective, and both $(Y,\Delta_Y)$ and $(Y^{0,c},\Delta_{Y^{0,c}})$ are demi-normal pairs with $\QQ$-Cartier log canonical divisors. Since the locus where $(Y,\Delta_Y)$ is semi-dlt is open, we can further assume the pair to be semi-dlt. Since $g$ is projective and there are only finitely many log centers lying on $Y\setminus Y^0$ (see, for instance, the proof of Lemma~\ref{lem: mld=1}), we can further restrict $S$ so that $Y^0$ intersects all log centers of $(Y,\Delta_Y)$ lying on the conductor. Similarly, we can assume that $\Ex(\psi^{0,c})$ contains no stratum of the conductor $D_Y$. Moreover, all properties in Definition~\ref{def: stable model slc pairs} are open conditions, hence we can also arrange $\psi^{0,c}$ to be a good stable model. Therefore, it is enough to prove the result for $T$ being affine and of finite type, which we will assume from now onward. The following argument resembles the proof of~\cite[Corollary 1.2]{Hacon_Xu_dslt}.

    Choose a dense open embedding $T\subset S$ with $S$ projective over $\CC$, and choose a compactification
    \[
    \begin{tikzcd}
        (X,\Delta_X)\arrow[d,"f"]\arrow[r,hookrightarrow] & (X',\Delta_{X'})\arrow[d,"g"]\\
        T\arrow[r,hookrightarrow] & S
    \end{tikzcd}
    \]
    with $X'$ reduced and $g$ projective. Eventually after blowing up some irreducible components of $X'\setminus X$ and taking the demi-normal modification (Definition~\ref{def: deminormal modification}), we can assume $X'$ to be demi-normal and $D_{X'}$ to be the closure of $D_X$. We can further assume $\Delta_{X'}$ to be the closure of $\Delta_X$, and $K_{X'}+\Delta_{X'}$ to be $\QQ$-Cartier. By applying the results of \cite{B-M-except-I,Bierstone-Vera,BDSMV} (see also~\cite[Theorems 10.58, 10.59]{Kollar-singularities}), we can construct a log semi-resolution $h\colon Y\rightarrow X'$ with
    \[
        (Y,h_*^{-1}\Delta_{X'}+E+\red(h^{-1}(X'\setminus X)))
    \]
    being semi-SNC, where $E$ is the reduced sum of the exceptional divisors, and such that $h$ induces birational morphisms between all strata of the conductors. Let $\Delta_Y:=h_*^{-1}\Delta_{X'}+\sum a_iE_i$, where the sum runs over the set of irreducible exceptional divisors $E_i$ such that $h(E_i)\not\subset X'\setminus X$ and such that $a_i=a(E_i|_X,X,\Delta_X)\geq0$. Let $Y_T:=T\times_SY$. By construction, $(Y,\Delta_Y)$ is semi-dlt and there are no log centers lying on $Y\setminus Y_T$. Moreover, on $Y_T$ we have
    \begin{equation}\label{eq: reduction to projective case}
        K_{Y_T}+\Delta_{Y_T}=h^*(K_X+\Delta_X)+F,
    \end{equation}
    where $F$ is effective, $h$-exceptional, and has no irreducible components in common with $\Delta_{Y_T}$. In particular, the restriction of $(Y,\Delta_Y)$ over $T^0$ admits a stable model over $T^0$, equal to $X^{0,c}$, and whose associated birational map is an isomorphism at the generic point of each stratum of the conductor $D_Y$. Similarly, the log canonical model of the normalization pair $(\Ybar,D_{\Ybar}+\Delta_{\Ybar})$ restricted to $T$ exists and it coincides with that of $(\Xbar,D_{\Xbar}+\Delta_{\Xbar})$. Since we noted that there are no log canonical centers on $\Ybar\setminus\Ybar_T$, the entire $\Ybar$ admits a log canonical model over $S$ by~\cite[Theorem 1.1]{Hacon_Xu_dslt}, applied in the case where the base is projective. Therefore, the datum $(Y,\Delta_Y)\rightarrow S$ satisfies all assumptions in Theorem~\ref{thm: general version} and $S$ is projective, thus $(Y,\Delta_Y)$ admits a stable model over $S$. Its restriction to $T$ is then the stable model of $(X,\Delta_X)$ over $T$.
\end{proof}

From now on, we may assume $T$ to be projective. To implement the perturbation strategy, the first step consists in showing that the perturbed pair $(X,\Delta_X+\epsilon A_X)$ satisfies the same assumptions as Theorem~\ref{thm: equivariant sklt}.
We start by showing that we can replace a perturbed boundary with a semi-klt one that is well-behaved with respect to any given birational map.

\begin{lemma}\label{lem: improving boundary by perturbation}
    Let $(X,\Delta_X)$ be a semi-dlt pair, $f\colon X\rightarrow T$ be a projective morphism to a projective variety $T$, and $A_X$ an $f$-ample $\QQ$-Cartier divisor. Let $\phi\colon X\dashrightarrow Y$ be a birational map over $T$, whose exceptional locus does not contain any stratum of $D_X$. Then, there exists an integer $c>0$ and a $\QQ$-divisor $B_X$ such that for all rational $0<\epsilon\ll1$ the following holds:
    \begin{enumerate}
        \item\label{req 1} on the open locus of $X$ where $X$ is SNC, we have that $B_X$ is $\QQ$-Cartier and $f$-ample, and $\Supp B_X$ is SNC-adapted;
        \item\label{req 2} $K_X+\Delta_X+\epsilon(cA_X)\sim_{\QQ,T}K_X+(1-\epsilon)\Delta_X+\epsilon B_X$;
        \item\label{req 3} $(X,(1-\epsilon)\Delta_X+\epsilon B_X)$ is semi-klt, in particular $\lfloor \epsilon B_X\rfloor=0$;
        \item\label{req 4} the exceptional locus of $\phi$ does not contain any component of $\Supp B_X$ and $\Supp B_X\cap W$ for any stratum $W$ of $D_X$ for which $\Supp B_X\cap W\not=\emptyset$.
    \end{enumerate}
\end{lemma}
\begin{proof}
    We follow the proof of~\cite[Proposition 2.43]{Kollar_Mori_1998} and adapt it to our setting. Choose positive integers $m$ and $m'$ such that $m\Delta_X$ is an integral divisor, $m'A_X$ is Cartier, and the (reflexive and locally free in codimension 1) sheaf $\cO_X(m\Delta_X+m'A_X)$ is generated by global sections over $T$.
    Set $c=m'/m$, where we can assume that $m$ divides $m'$. Let $B_X'\in|m\Delta_X+m'A_X|_T$ be a general member, and let $B_X=(1/m)B_X'$.
    Then, $-m\Delta_X+mB_X\sim_T m'A_X$ is Cartier, thus
    \[
        K_X+(1-\epsilon)\Delta_X+\epsilon B_X\sim_{\QQ,T} K_X+\Delta_X+\epsilon(c A_X)
    \]
    and is $\QQ$-Cartier. Let $Z\subset X$ be the closed subscheme in the definition of semi-dlt pair (Definition~\ref{def: dlt/sdlt/sklt}), that is, the complement $U=X\setminus Z$ is SNC, $\Delta_X|_U$ is SNC-adapted to $X$, and the discrepancy over any center in $Z$ is strictly greater than $-1$. Note that $U$ is contained in the larger open $U'$ where $X$ is SNC, and in particular Gorenstein. Therefore, on $U'$ the divisor $\Delta_X$ is $\QQ$-Cartier, thus the same holds for $B_X$. Choosing $m'$ big enough, we can assume that $B_X$ is $T$-ample on $U'$. Moreover, by base-point freeness (relatively to $T$) and by choosing $B_X$ general, on $U'$ the support $\Supp B_X$ is SNC-adapted to $X$, hence it has simple normal crossing with $D_X$.
    Then, for $0\leq\epsilon\ll1$ we have that the perturbed pair is semi-dlt on $U'$, hence semi-klt as soon as $\lfloor\epsilon B_X\rfloor=0$. On $Z$, any sufficiently small perturbation stays slc and does not create additional log canonical centers, thus $(X,(1-\epsilon)\Delta_X+\epsilon B_X)$ is everywhere semi-klt for $0\leq\epsilon\ll1$.

    We are left with showing that $\Ex(\phi)$ does not contain any component of $B_X$ and $B_X\cap W$ for any stratum $W$ of $D_X$. On the other hand, this follows from the global generation of $\cO_X(m\Delta_X+m'A_X)$ over $T$, as we just need to choose $B_X$ so that it avoids the generic points of $\Ex(\phi)$ and $\Ex(\phi)\cap W$.
\end{proof}

The following lemma shows that the perturbed pair still admits a stable model over $T^0$.

\begin{lemma}\label{lem: existence stable model perturbed pair}
    Let $(X,\Delta_X)$ be a semi-dlt pair and $f\colon X\rightarrow T$ a projective morphism to a projective variety $T$. Suppose the pair admits a good stable model $(X^c,\Delta_{X^c})$ over $T$, and that the exceptional locus of $\phi^c\colon X\dashrightarrow X^c$ does not contain any stratum of the conductor $D_X$. Let $A_X$ be an $f$-ample $\QQ$-Cartier divisor on $X$. Then, the pair $(X,\Delta_X+\epsilon A_X)$ admits a good stable model $\phi_{\epsilon}^c\colon X\dashrightarrow X_{\epsilon}^c$ for all $0<\epsilon\ll 1$, and $\Ex(\phi_{\epsilon}^c)$ does not contain any stratum of $D_X$.
\end{lemma}

\begin{proof}
    By Lemma~\ref{lem: improving boundary by perturbation}, there is an integer $c>0$ and a boundary $\QQ$-divisor $B_X$ such that $K_X+\Delta_X+\epsilon(cA_X)\sim_{\QQ,T}K_X+(1-\epsilon)\Delta_X+\epsilon B_X$ and $(X,(1-\epsilon)\Delta_X+\epsilon B_X)$ is semi-klt for all $0<\epsilon\ll1$. In particular, the normalization pair $(\Xbar,D_{\Xbar}+(1-\epsilon)\Delta_{\Xbar}+\epsilon B_{\Xbar})$ is dlt.
    
    Moreover, the normalization pair admits a log canonical model by~\cite{BCHM} and~\cite[Proposition 2.43]{Kollar_Mori_1998}.
    Eventually after taking a small $\QQ$-factorial modification, the MMP with scaling of the ample divisor $A_{\Xbar}:=\pi_X^*A_X$ terminates and the resulting $\QQ$-factorial minimal model $\phibar^m\colon \Xbar\dashrightarrow\Xbar^m$ of $(\Xbar/T,D_{\Xbar}+\Delta_{\Xbar}+\epsilon(cA_{\Xbar}))$ does not depend on $0\leq\epsilon\ll1$; see for instance~\cite[Section 2]{Hacon_Xu_dslt}. Note in particular that the minimal model is the same as that of the original pair. We set $B_{\Xbar^m}:=\phibar^m_*B_{\Xbar}$. Let us also denote by $R_{D_{\Xbar^m}}^{\epsilon}$ the log canonical algebra of the perturbed pair $(D_{\Xbar^m},\Diff_{D_{\Xbar^m}}(\Delta_{\Xbar^m})+\epsilon B_{\Xbar^m}|_{D_{\Xbar^m}})$, and similarly for all other perturbed pairs.
    
    Since $(X,\Delta_X)$ admits a stable model $\phi^c\colon X\dashrightarrow X^c$ over $T$, there is an involution on the pair $(D_{\Xbar^c}^n,\Diff_{D_{\Xbar^c}^n}(\Delta_{\Xbar^c}))$ that defines the gluing relations. Since $\phibar_{D^n}\colon D_{\Xbar^m}^n\rightarrow D_{\Xbar^c}^n$ is birational by the assumption on the exceptional locus of $\phi^c$, and the morphism is crepant (see the proof of Lemma~\ref{lem: conductor for sklt induction} or~\cite[Proposition 4.6]{Kollar-singularities}), $\phibar_{D^n}$ is a crepant log canonical model over $T$. By Remark~\ref{rmk: Bir is a crepant invariant}, the involution on $D_{\Xbar^c}^n$ induces a B-birational involution $\mu\in\Bir(D_{\Xbar^m}^n/T,\Diff_{D_{\Xbar^m}^n}(\Delta_{\Xbar^m}))$. Since $(D_{\Xbar^m}^n,\Diff_{D_{\Xbar^m}^n}(\Delta_{\Xbar^m}))$ is a minimal model, $\mu$ is a regular automorphism on an open $U\subset D_{\Xbar^m}^n$ whose complement is of codimension $\geq2$, see for instance~\cite[Theorem 3.52, Corollary 3.54]{Kollar_Mori_1998}. In particular, to check that $\mu\in\Bir(D_{\Xbar^m}^n/T,M)$ for some boundary $M$, it is enough to check so for $\mu|_U$.

    By Lemma~\ref{lem: improving boundary by perturbation}\eqref{point 4} applied to $\phi^c$, none of the components of the divisor $B_{\Xbar}|_{D_{\Xbar}}$ is contained in the exceptional locus of the birational map $\phibar^m_D\colon D_{\Xbar}\dashrightarrow D_{\Xbar^m}$. In particular, $\phi^c$ is an isomorphism at the generic points of $B_{\Xbar}|_{D_{\Xbar}}$ and maps it birationally onto $B_{\Xbar^m}|_{D_{\Xbar^m}}$. Since the diagram
    \[
    \begin{tikzcd}
        D_{\Xbar}^n\arrow[r,dashed]\arrow[d,"\simeq","\text{gluing}"'] & D_{\Xbar^m}^n\arrow[d,dashed,"\mu"]\\
        D_{\Xbar}^n\arrow[r,dashed] & D_{\Xbar^m}^n
    \end{tikzcd}
    \]
    is commutative and $\mu$ is a regular automorphism in codimension 1, it follows that
    \[
        \mu\in\Bir(D_{\Xbar^m}^n/T,\Diff_{D_{\Xbar^m}^n}(\Delta_{\Xbar^m})+\epsilon B_{\Xbar^m}|_{D_{\Xbar^m}}).
    \]
    By Lemma~\ref{lem: conductor for sklt induction}\eqref{point 3}, the pair $(D_{\Xbar^m},\Diff_{D_{\Xbar^m}}(\Delta_{\Xbar^m})+\epsilon B_{\Xbar^m}|_{D_{\Xbar^m}})$ is semi-klt, and it admits a stable model over $T$ by Lemma~\ref{lem: birational coeff 1 boundary then stable model}. Therefore, we can apply Theorem~\ref{thm: equivariant sklt} to obtain that the inclusions
    \[
    \begin{tikzcd}
        R_{D_{\Xbar^m}}^{\epsilon}\cap (R_{D_{\Xbar^m}^n}^{\epsilon})^{\langle\mu\rangle}\arrow[r,hookrightarrow] & R_{D_{\Xbar^m}}^{\epsilon}\arrow[r,hookrightarrow] & R_{D_{\Xbar^m}^n}^{\epsilon}
    \end{tikzcd}
    \]
    are finite extensions of $\cO_T$-algebras, where we are using the convention of Definition~\ref{def: invariants under Bir}. By Lemma~\ref{lem: surjectivity restriction minimal model}, the restriction morphism $\res\colon R_{\Xbar^m}^{\epsilon}\rightarrow R_{D_{\Xbar^m}}^{\epsilon}$ induces a surjection between $\ell$-Veronese subalgebras for $\ell$ sufficiently divisible. It follows that
    \[
    \begin{tikzcd}
        \res^{-1}(R_{D_{\Xbar^m}}^{\epsilon}\cap (R_{D_{\Xbar^m}^n}^{\epsilon})^{\langle\mu\rangle})\arrow[r,hookrightarrow] & R_{\Xbar}^{\epsilon}
    \end{tikzcd}
    \]
    is also an integral extension, hence finite as $R_{\Xbar}^{\epsilon}$ is of finite type. By Lemma~\ref{lem: log canonical ring slc pair} and Remark~\ref{rmk: different invariants mu}, the inverse image on the left is identified with a Veronese subalgebra of even degree of $R_X^{\epsilon}$. Therefore, $R_X^{\epsilon}\hookrightarrow R_{\Xbar}^{\epsilon}$ is also a finite extension, hence $R_X^{\epsilon}$ is finitely generated for all $0\leq\epsilon\ll1$.

    Since $\phi^c$ is an isomorphism at the generic point of any stratum of $D_X$, and $\phi_{\epsilon}^c$ is induced by twisting with an $f$-ample Cartier divisor, the same property holds for $\phi_{\epsilon}^c$. In particular, it is a good stable model by Theorem~\ref{thm: projRX is stable model}.
\end{proof}

We also need to choose the perturbation so that it does not create new log centers on $D_X$ that map to $T\setminus T^0$. For this, we will use the following simple lemma characterizing centers of SNC pairs whose minimal log discrepancy is 1 (\S\ref{subsubsec: conventions discrepancy and centers}) and that do not dominate $T$.

\begin{lemma}\label{lem: mld=1}
    Let $Y$ be a smooth connected variety and let $\Delta_Y=\sum_{i\in I}a_i\Delta_i$ be an SNC divisor such that $a_i\leq 1$ and are non-zero for all $i\in I$. Let $\Delta_Y^{>0}=\sum_{i\in I:a_i>0}a_i\Delta_i$ be the positive part of $\Delta_Y$, and $f\colon Y\rightarrow T$ a proper surjective morphism such that all strata of $\Delta_Y^{>0}$ surject onto $T$.
    Let $W\subset Y$ be an irreducible subvariety, whose image $f(W)=Z\subsetneq T$ is a proper subvariety of $T$. If $\mathrm{mld}(W,Y,\Delta_Y)=1$, then one of the following is satisfied:
    \begin{enumerate}
       \item\label{discrepancy 0: case 1} There exists $J\subset I$ such that $W$ is an irreducible component of $\cap_{j\in J}\Delta_j$, and $\sum_{j\in J}a_j=|J|-1$. Moreover, in this case there exists $j\in J$ for which $a_j<0$.
       \item\label{discrepancy 0: case 2} There exists $J\subset I$ with $a_j=1$ for all $j\in J$, such that $W$ is an irreducible component of $(\cap_{j\in J}\Delta_j)\cap f^{-1}(Z)$ and has codimension 1 in $(\cap_{j\in J}\Delta_{j})$.
    \end{enumerate}
    In particular, there exists at most a finite number of such centers $W$.
\end{lemma}
\begin{proof}
    Let $J$ be the subset of indices $j\in I$ for which $W\subset\Delta_j$. By~\cite[2.10]{Kollar-singularities}, we know that
    \begin{equation}\label{eq: equality mld}
        1=\mathrm{mld}(W,Y,\Delta_Y)=\codim_{Y}W-\sum_{j\in J}a_j\geq\codim_{Y}W-|J|,
    \end{equation}
    as $a_j\leq1$. In particular, $|J|+1\geq\codim_{Y}W\geq|J|$.
    
    If $\codim_{Y}W=|J|$, then $W$ is an irreducible component of $\cap_{j\in J}\Delta_j$, hence a stratum of $\Delta_Y$. Since $f(W)\subsetneq T$ and all strata of $\Delta_Y^{>0}$ dominate $T$, it follows that there exists $j\in J$ for which $a_j<0$. The equality $\sum_{j\in J}a_j=|J|-1$ follows from~\eqref{eq: equality mld}.
    
    Suppose now that $\codim_{Y}W=|J|+1$. From~\eqref{eq: equality mld} and the assumption $a_j\leq1$, we get that  $a_j=1$ for all $j\in J$, hence $\Delta_j$ is a component of $\Delta_Y^{>0}$ for all $j\in J$. By assumption, $W\subseteq(\cap_{j\in J}\Delta_j)\cap f^{-1}(Z)$, and $\cap_{j\in J}\Delta_j$ dominates $T$. Therefore, $W$ is an irreducible component of $(\cap_{j\in J}\Delta_j)\cap f^{-1}(Z)$ and has codimension 1 in $\cap_{j\in J}\Delta_j$.
\end{proof}

We are ready to prove the main result of the paper.

\begin{proof}[Proof of Theorem~\ref{thm: general version}]
    By Theorem~\ref{thm: equivariant sklt}, we know that the statement holds for all semi-klt pairs. Moreover, by Lemma~\ref{lem: reduction to projective case} we can assume that $T$ is a projective variety. We can also assume $X$ to be non-normal, hence $D_X\not=\emptyset$.

    Fix an $f$-ample line bundle $A_X$ on $X$, and denote by $A_{\Xbar}$ its pullback to $\Xbar$. By Lemma~\ref{lem: improving boundary by perturbation} applied to the birational map $\phi^c\colon X\dashrightarrow X^{0,c}$, there exists a boundary $\QQ$-divisor $B_X$ such that
    \[
        K_X+(1-\epsilon)\Delta_X+\epsilon B_X\sim_{\QQ,T}K_X+\Delta_X+\epsilon(cA_X)
    \]
    for every rational $0<\epsilon\ll1$, and $(X,(1-\epsilon)\Delta_X+\epsilon B_X)$ is semi-klt. Lemma~\ref{lem: existence stable model perturbed pair} says that the restriction to $X^0$ of the perturbed pair still admits a stable model $\phi_{\epsilon}^{0,c}\colon X^0\dashrightarrow X_{\epsilon}^{0,c}$ such that $\Ex(\phi_{\epsilon}^{0,c})$ does not contain any stratum of $D_{X^0}$.
    To apply Theorem~\ref{thm: equivariant sklt}, we need to show that we can choose $B_X$ so that $(X,(1-\epsilon)\Delta_X+\epsilon B_X)$ does not admit new log centers on $D_X$ that do not intersect $X^0$. Since the discrepancies are rational with denominator bounded by the index (\cite[Definition 2.49]{Kollar-singularities}) of the log canonical divisor and $0<\epsilon\ll1$, it is enough to show that we can choose $B_X$ so that its support does not contain centers $W$ lying on $D_X$ with $\mathrm{mld}(W,X,\Delta_X)=1$ and mapping to $T\setminus T^0$. By taking a log resolution of the normalization of $(X,\Delta_X)$ and applying Lemma~\ref{lem: mld=1} to a neighborhood of $D_{\Xbar}$ in each connected component, we see that such centers $W$ are in finite number. In particular, in Lemma~\ref{lem: improving boundary by perturbation} we can choose $B_X$ not containing them.
    
    Therefore, for all $0<\epsilon\ll1$, we can apply Theorem~\ref{thm: equivariant sklt} to produce good stable models
    \[
    \begin{tikzcd}
        \phi_{\epsilon}^c\colon (X,(1-\epsilon)\Delta_X+\epsilon B_X)\arrow[r,dashed] & (X_{\epsilon}^c,(1-\epsilon)\Delta_{X_{\epsilon}^c}+\epsilon B_{X_{\epsilon}^c}),
    \end{tikzcd}
    \]
    whose normalization is a log canonical model
    \[
    \begin{tikzcd}
        \phibar_{\epsilon}^c\colon (\Xbar,D_{\Xbar}+(1-\epsilon)\Delta_X+\epsilon B_{\Xbar})\arrow[r,dashed] & (\Xbar_{\epsilon}^c,D_{\Xbar_{\epsilon}^c}+(1-\epsilon)\Delta_{\Xbar_{\epsilon}^c}+\epsilon B_{\Xbar_{\epsilon}^c}).
    \end{tikzcd}
    \]
    Moreover, by assumption the dlt pair $(\Xbar,D_{\Xbar}+\Delta_{\Xbar})$ admits a good canonical model $\phibar^c_0=\phibar^c\colon \Xbar\dashrightarrow\Xbar_0^c=\Xbar^c$ and a good minimal model. By~\cite[Corollary 2.9]{Hacon_Xu_dslt}, eventually after taking a small $\QQ$-factorial modification, the $(K_{\Xbar}+D_{\Xbar}+\Delta_{\Xbar})$-MMP over $T$ with $A_{\Xbar}$-scalings terminates with a birational contraction over $T$
    \[
    \begin{tikzcd}
        \phibar^m\colon \Xbar\arrow[r,dashed] & \Xbar^m
    \end{tikzcd}
    \]
    that is a $\QQ$-factorial $K_{\Xbar}+D_{\Xbar}+\Delta_{\Xbar}+\epsilon(cA_{\Xbar})$-good minimal model for all $0\leq \epsilon\ll1$.
    We emphasize that this includes the case $\epsilon=0$. As usual, set $\Delta_{\Xbar^m}=\phibar^m_*\Delta_{\Xbar}$, and similarly for all the other divisors.
    Since $\Xbar^m$ is a good minimal model, for all $0\leq\epsilon\ll1$ there exist morphisms
    \[
    \begin{tikzcd}
        \phibar_{\epsilon}\colon \Xbar^m\arrow[r] & \Xbar_{\epsilon}^c
    \end{tikzcd}
    \]
    to the respective stable models through which $\phibar_{\epsilon}^c$ factors.
    
    We claim that $\phibar_{\epsilon}$ and $\Xbar_{\epsilon}^c$ do not depend on $\epsilon$ for $0<\epsilon\ll1$; note that these could (and usually do) change at $\epsilon=0$.
    To prove this, recall that, by definition, $\phibar_{\epsilon}$ can be characterized as the birational contraction that contracts the curves $C$ that are vertical over $T$ and satisfy
    \[
        C\cdot(K_{\Xbar^m}+D_{\Xbar^m}+(1-\epsilon)\Delta_{\Xbar^m}+\epsilon B_{\Xbar^m})=C\cdot(K_{\Xbar^m}+D_{\Xbar^m}+\Delta_{\Xbar^m}+\epsilon(cA_{\Xbar^m}))=0.
    \]
    On the other hand, $C\cdot(K_{\Xbar^m}+D_{\Xbar^m}+\Delta_{\Xbar^m})\geq0$, thus the above equation would imply that $\epsilon(cA_{\Xbar^m})\cdot C<0$. Choosing $0<\epsilon<\epsilon'\ll1$, we would then get
    \[
       C\cdot(K_{\Xbar^m}+D_{\Xbar^m}+\Delta_{\Xbar^m}+\epsilon'(cA_{\Xbar^m}))<C\cdot(K_{\Xbar^m}+D_{\Xbar^m}+\Delta_{\Xbar^m}+\epsilon(cA_{\Xbar^m}))=0,
    \]
    contradicting the nefness of $K_{\Xbar^m}+D_{\Xbar^m}+\Delta_{\Xbar^m}+\epsilon'(cA_{\Xbar^m})$. The independence of $\phibar_{\epsilon}^c$ from $0<\epsilon\ll1$ follows.
    We denote by
    \[
    \begin{tikzcd}
        \phibar_+\colon \Xbar^m\arrow[r] & \Xbar_+^c
    \end{tikzcd}
    \]
    the canonical morphism equal to $\phibar_{\epsilon}$ for all $0<\epsilon\ll1$.
    
    Of course, in general $\Xbar_0^c\not=\Xbar_+^c$, as one could drop into a smaller face of the nef cone. On the other hand, if $C$ is contracted by $\phibar_+$, then
    \begin{align*}
        0&=C\cdot\left((K_{\Xbar^m}+D_{\Xbar^m}+\Delta_{\Xbar^m}+\epsilon(cA_{\Xbar^m}))-(K_{\Xbar^m}+D_{\Xbar^m}+\Delta_{\Xbar^m}+\epsilon'(cA_{\Xbar^m})\right))\\
        &=(\epsilon-\epsilon')C\cdot A_{\Xbar^m},
    \end{align*}
    for all $0<\epsilon<\epsilon'\ll1$. Therefore, $C\cdot(K_{\Xbar^m}+D_{\Xbar^m}+\Delta_{\Xbar^m})=0$, and $C$ is contracted by $\phibar_0$ as well. In particular, $\phibar_0$ factors through a birational contraction
    \[
    \begin{tikzcd}
        \psi\colon \Xbar_+^c\arrow[r] & \Xbar_0^c
    \end{tikzcd}
    \]
    by the Rigidity Lemma~\cite[Lemma 1.15]{Debarre-higher_dimensional_AG}.
    Note that the log canonical model $\phibar_+^c:=\phibar_+\circ\phibar^m\colon \Xbar\dashrightarrow\Xbar_+^c$ does not contract any stratum of the conductor $D_{\Xbar}$. Indeed, otherwise $\phibar^c=\psi\circ\phibar_+^c$ would also contract some stratum of $D_{\Xbar}$, contradicting the assumption that $\phi^{0,c}$ does not contract any stratum of the conductor $D_X$ of $X$.
    
    We claim that $K_{\Xbar_+^c}+D_{\Xbar_+^c}+\Delta_{\Xbar_+^c}$ is semi-ample over $T$.
    To prove this, it is enough to show that $\psi^*(K_{\Xbar_0^c}+D_{\Xbar_0^c}+\Delta_{\Xbar_0^c})\sim_{\QQ,T}K_{\Xbar_+^c}+D_{\Xbar_+^c}+\Delta_{\Xbar_+^c}$. On the other hand,
    \begin{align*}
        \psi^*(K_{\Xbar_0^c}+D_{\Xbar_0^c}+\Delta_{\Xbar_0^c}) & \sim_{\QQ,T} \phi_{+*}\phi_+^*\psi^*(K_{\Xbar_0^c}+D_{\Xbar_0^c}+\Delta_{\Xbar_0^c})\\
        & =\phi_{+*}\phi_0^*(K_{\Xbar_0^c}+D_{\Xbar_0^c}+\Delta_{\Xbar_0^c})\\
        & = \phi_{+*}(K_{\Xbar^m}+D_{\Xbar^m}+\Delta_{\Xbar^m})\\
        & \sim_{\QQ,T} K_{\Xbar_+^c}+D_{\Xbar_+^c}+\Delta_{\Xbar_+^c}
    \end{align*}
    where the first $\QQ$-linear equivalence follows from the projection formula, and the last by definition. This proves that $K_{\Xbar_+^c}+D_{\Xbar_+^c}+\Delta_{\Xbar_+^c}$ is semi-ample over $T$.

    Now, we work at the demi-normal level, and we claim that also $\phi_{\epsilon}^c\colon X\dashrightarrow X_{\epsilon}^c$ is independent of $\epsilon$ as long as $0<\epsilon\ll1$. From what we have proved above, the normalization $\phibar_{+}^c\colon \Xbar\dashrightarrow\Xbar_+^c$ of $\phi_{\epsilon}^c$ is independent of $\epsilon$ in the same range, and it restricts to a birational map between conductors, thus also the rational involution on $D_{\Xbar_+^c}$ is independent of $\epsilon$. Since a demi-normal pair is uniquely determined by its normalization triple (see Remark~\ref{rmk: normalization pair} or \cite[Proposition 5.3]{Kollar-singularities}), $X\dashrightarrow X_{\epsilon}^c$ is constant for $0<\epsilon\ll1$, as wanted. As before, we denote $X_{\epsilon}^c$ by $X_+^c$ for $0<\epsilon\ll1$, and similarly for the relevant divisors.
    
    Recall that we also know
    \[
    \begin{tikzcd}
        \pi_+\colon (\Xbar_+^c,D_{\Xbar_+^c}+(1-\epsilon)\Delta_{\Xbar_+^c}+\epsilon B_{\Xbar_+^c})\arrow[r] & (X_+^c,(1-\epsilon)\Delta_{X_+^c}+\epsilon B_{X_+^c})
    \end{tikzcd}
    \]
    to be the normalization. It follows that
    \[
        \pi_+^*(K_{X_+^c}+\Delta_{X_+^c})\sim_{\QQ,T}K_{\Xbar_+^c}+D_{\Xbar_+^c}+\Delta_{\Xbar_+^c},
    \]
    which is semi-ample. By~\cite[Theorem 1.4]{FG_fin_Brepr}, \cite{FujinoAbundanceSLC}, and~\cite[Remark 1.5]{Hacon_Xu_finBrepr_slc_abundance}, we obtain that $K_{X_+^c}+\Delta_{X_+^c}$ is semi-ample as well. Therefore, the associated $\cO_T$-algebra $R(X_+^c/T,K_{X_+^c}+\Delta_{X_+^c})\simeq R_X$ is of finite type, and $X^c:=\Proj R_X$ with the induced boundary is the good stable model of $(X,\Delta_X)$ by Theorem~\ref{thm: projRX is stable model}. Note that there could be vertical irreducible components of $X$ over $T$, but they have image in $T^0$, where we already know $\Proj R_X$ to be the stable model by assumption.
\end{proof}

\subsection{Proof of Corollaries~\ref{cor: main intro} and~\ref{cor: simple normal crossing case}}\label{subsec:corollaries}

In this subsection we prove Corollary~\ref{cor: main intro} and a special case of it.

\begin{proof}[Proof of Corollary~\ref{cor: main intro}]
    To prove Corollary~\ref{cor: main intro}, it is enough to show that $(\Xbar,D_{\Xbar}+\Delta_{\Xbar})$ admits a log canonical model over $T$, and then apply Theorem~\ref{thm: general version}. First, the same argument as Lemma~\ref{lem: reduction to projective case} allows us to reduce the statement to the case where $T$ is projective. Then, the result follows from~\cite[Theorem 1.1]{Hacon_Xu_dslt}.
\end{proof}

\begin{remark}\label{rmk: reduction projective case for Hacon-Xu result}
    In the above proof, the reduction to $T$ being projective was not strictly necessary in order to apply~\cite[Theorem 1.1]{Hacon_Xu_dslt}. However, as explained in \S\ref{subsec: motivation}, Hacon and Xu use Koll\'ar's gluing to establish a semi-ampleness result~\cite[Proposition 3.1]{Hacon_Xu_dslt}. When the base is projective, this is a consequence of more general results~\cite[Theorem 1.4]{FG_fin_Brepr}, \cite{FujinoAbundanceSLC}.
\end{remark}

As a special case, we obtain Corollary~\ref{cor: simple normal crossing case}.

\begin{proof}[Proof of Corollary~\ref{cor: simple normal crossing case}]
    Note that $(X,\Delta_X)$ is semi-dlt, and it is an slc family~\cite[Definition/Theorem 2.3]{Kollar-moduli}. In particular, all log centers (thus also all lc centers) intersect $X^0:=X|_{T^0}$ by~\cite[Proposition 2.15]{Kollar-moduli}. Moreover, by assumption, the restriction $\phi^0\colon X^0\rightarrow Y^0$ of $\phi$ does not contract any stratum of the conductor $D_{X^0}$, and it is a stable model over $T^0$. Therefore, we can apply Corollary~\ref{cor: main intro} to obtain the stable model $(X^c,\Delta_{X^c})$ over $T$ as in the statement. Now, for all $t\in T$, the pair $(X,\Delta_X+X_t)$ is again slc and $X_t$ is relatively trivial over $T$, hence its stable model is $(X^c,\Delta_{X^c}+(X^c)_t)$ by Theorem~\ref{thm: projRX is stable model}, which is then slc. Again by~\cite[Definition/Theorem 2.3]{Kollar-moduli}, it follows that $(X^c,\Delta_{X^c})$ is a stable slc family over $T$.
\end{proof}

\section{MMP for slc pairs}\label{sec:slc-MMP}

In this section we prove Corollary~\ref{cor: analog Birkar} and outline how to deduce the existence of certain MMP steps for slc pairs without relying on Koll\'ar's gluing technique.

\begin{proof}[Proof of Corollary~\ref{cor: analog Birkar}]
    By~\cite[Theorem 1.1(3)]{Birkar_existenceflips} or~\cite[Theorem 1.6]{Hacon_Xu_dslt}, the normalization pair $(\Xbar,D_{\Xbar}+\Delta_{\Xbar})$ admits both a good minimal model and a log canonical model over $T$. Taking the Stein factorization of $f$, we may further that $f$ is birational, is an isomorphism at the generic point of every log center lying on $D_X$, and satisfies $f_*\cO_X\simeq\cO_T$. Set $X^0=X\setminus\Ex(f)$ and $T^0=T\setminus f(\Ex(f))$, thus $X^0$ intersects all log centers lying on $D_X$. By connectedness of the fibers, we have $X^0=f^{-1}(T^0)\simeq T^0$.
    Therefore, all assumptions of Theorem~\ref{thm: general version} are satisfied, hence $(X,\Delta_X)$ admits a stable model over $T$.
\end{proof}

Let us recall a definition of MMP steps for slc pairs.

\begin{definition}[{\cite[Definition 11]{Ambro-Kollar-slc-MMP}, \cite[Definition 2.6]{Hashizume_sdlt_models}}]\label{def:slc-MMP steps}
    Let $(X,\Delta_X)$ be an slc pair, and let $f\colon X\rightarrow T$ be a projective morphism with $T$ a quasi-projective scheme. Then, an \emph{MMP step} over $T$ for $(X,\Delta_X)$ is a commutative diagram
    \[
    \begin{tikzcd}
        (X,\Delta_X)\arrow[rr,dashed,"\psi"]\arrow[rd,"\varphi"'] & & (Y,\Delta_Y)\arrow[ld,"\varphi'"]\\
        & Z
    \end{tikzcd}
    \]
    over $T$, such that
    \begin{enumerate}
        \item\label{slc_step:1} $(Y,\Delta_Y)$ is an slc pair;
        \item\label{slc_step:2} $\psi$ is birational and $\Ex(\psi)$ does not contain any component of the conductor;
        \item $\Delta_Y=\psi_*\Delta_X$;
        \item $Z$ and $Y$ are quasi-projective schemes, the morphisms $Z\rightarrow T$ and $Y\rightarrow T$ are projective, $\varphi$ and $\varphi'$ are generically finite, and $\varphi'$ has no exceptional divisors;
        \item $-(K_X+\Delta_X)$ and $(K_Y+\Delta_Y)$ are ample over $Z$.
    \end{enumerate}
\end{definition}

\begin{remark}\label{rmk: normalization MMP step}
    By~\cite[Lemma 12]{Ambro-Kollar-slc-MMP}, a diagram as in Definition~\ref{def:slc-MMP steps} satisfying properties~\eqref{slc_step:1} and~\eqref{slc_step:2} is an MMP step over $T$ for $(X,\Delta_X)$ if and only if the diagram obtained by normalization is an MMP step.
\end{remark}

\subsection{MMP for semi-dlt pairs}

We first explain Corollary~\ref{cor: analog Birkar} yields the existence of certain MMP steps for semi-dlt pairs.

The Cone Theorem is known for all slc pairs by Fujino~\cite[Theorem 1.19(3)]{fujino_theorems_slc}. However, in general, flips need not exist~\cite[Examples 4, 5]{Ambro-Kollar-slc-MMP}, and after a divisorial contraction the resulting pair need not have $\QQ$-Cartier log canonical divisor, even though its pullback to the normalization does have the property, see~\cite[Example 5.4]{fujino_theorems_slc}. We show that these pathologies do not occur when the exceptional locus of the contraction contains no log center lying on the conductor, in which case we can construct an MMP step. This gives a direct alternative to the corresponding constructions obtained using Koll\'ar's gluing theory in~\cite[Theorems 7 and 9]{Ambro-Kollar-slc-MMP} and~\cite[Theorem 5.38]{Kollar-singularities}; see also~\cite[Lemmas 3.1, 3.4]{Hashizume_sdlt_models}.

\begin{corollary}\label{cor: existence semidlt MMP}
    Let $(X,\Delta_X)$ be a semi-dlt pair, and $f\colon X\rightarrow T$ a projective morphism to a semi-normal quasi-projective variety $T$. Let $F\subset\overline{NE}(X/T)$ be a $(K_X+\Delta_X)$-negative extremal ray, and $\varphi=\varphi_F\colon X\rightarrow Z$ the corresponding contraction. Suppose that $\varphi$ is birational and $\Ex(\varphi)\cap D_X$ contains no log center. Then, $Z$ is demi-normal and there exists an MMP step
    \[
    \begin{tikzcd}
        (X,\Delta_X)\arrow[rr,dashed,"\psi"]\arrow[rd,"\varphi"'] & & (Y,\Delta_Y)\arrow[ld,"\varphi'"]\\
        & Z
    \end{tikzcd}
    \]
    over $T$. Moreover, $(Y,\Delta_Y)$ is semi-dlt.
\end{corollary}
\begin{proof}
    Note that $\varphi$ does not contract any component of the conductor. Since $\varphi_*\cO_X\simeq\cO_Z$ by~\cite[Theorem 1.19(3)]{fujino_theorems_slc}, we have that $Z$ is demi-normal by Example~\ref{exm: deminormal modifications}\eqref{exm: deminormal stein factorization}. By~\cite[Theorems 1.18, 1.19]{fujino_theorems_slc}, $Z$ is projective over $T$ and $-(K_X+\Delta_X)\sim_{\QQ,Z}A$ for some $\varphi$-ample $\QQ$-Cartier $\QQ$-divisor $A$. By ampleness, we can choose $A$ so that $(X,\Delta_X+A)$ is slc. By Corollary~\ref{cor: analog Birkar}, there exists a stable model $\psi\colon (X,\Delta_X)\dashrightarrow(Y,\Delta_Y)$ over $Z$ satisfying Definition~\ref{def:slc-MMP steps}\eqref{slc_step:1}-\eqref{slc_step:2}, and whose normalization $\psibar$ is the log canonical model of $(\Xbar,D_{\Xbar}+\Delta_{\Xbar})$ over $Z$. Moreover, we know from the classical theory that $\psibar$ is an MMP step, hence the same holds for $\psi$ by Remark~\ref{rmk: normalization MMP step}. The fact that $(Y,\Delta_Y)$ is semi-dlt is proved in~\cite[Lemma 3.4]{Hashizume_sdlt_models} in a more general setting, but in the present case the argument is simpler. Since $(Y,\Delta_Y)$ is dlt away from the conductor, it is enough to show that $\psi$ is an isomorphism at the generic point of every log canonical center of $(X,\Delta_X)$ lying on $D_X$. This follows from the stronger assumption that $\Ex(\varphi)\cap D_X$ contains no log center.
\end{proof}

\begin{remark}\label{rmk: semidlt divisorial contraction}
    With the assumptions of Corollary~\ref{cor: existence semidlt MMP}, suppose further that the restriction of the normalization $\overline{\varphi}\colon \Xbar\rightarrow\Zbar$ to every connected component of $\Xbar$ is a divisorial contraction. Then, $\overline{\varphi}'\colon \Ybar\rightarrow\Zbar$ is an isomorphism. Since $\varphi$ is an isomorphism at all the generic points of the conductor, $\varphi'\colon Y\rightarrow Z$ is an isomorphism in codimension 1. Therefore, $\varphi'$ is an isomorphism by Remark~\ref{rmk: normalization pair}, hence $(Z,\varphi_*\Delta_X)$ is semi-dlt and $\varphi$ the MMP step.
\end{remark}

Corollary~\ref{cor: existence semidlt MMP} allows us to run some steps of the MMP for semi-dlt pairs, under the assumption that the relative stable locus of $K_X+\Delta_X$ over $T$ contains no log center lying on the conductor. This recovers, in the semi-dlt setting, the corresponding results of~\cite[Lemma 3.1]{Hashizume_sdlt_models} without invoking Koll\'ar's gluing theory. In particular, the same argument as in Hashizume's paper gives crepant semi-dlt modifications as in~\cite[Theorem 1.2]{Hashizume_sdlt_models}.
Thus,~\cite[Theorem 1.2]{Hashizume_sdlt_models} is available to us in our study of the MMP for slc pairs without gluing along finite equivalence relations.

\subsection{MMP for slc pairs}

The existence of semi-dlt modifications allows us to reduce many questions about slc pairs to the semi-dlt case, similarly to dlt-modifications in the normal setting. Although standard, we describe how to extend Corollary~\ref{cor: existence semidlt MMP} to slc pairs for the sake of completeness.

\begin{corollary}\label{cor: existence slc MMP}
    Let $(X,\Delta_X)$ be an slc pair, and $f\colon X\rightarrow T$ a projective morphism to a semi-normal quasi-projective variety $T$. Let $F\subset\overline{NE}(X/T)$ be a $(K_X+\Delta_X)$-negative extremal ray, and $\varphi=\varphi_F\colon X\rightarrow Z$ the corresponding contraction. Suppose that $\varphi$ is birational and $\Ex(\varphi)\cap D_X$ contains no log center. Then, $Z$ is demi-normal and there exists an MMP step over $T$.
\end{corollary}
\begin{proof}
    Since $\varphi$ does not contract any component of the conductor and $\varphi_*\cO_X\simeq\cO_Z$ by~\cite[Theorem 1.19(3)]{fujino_theorems_slc}, we have that $Z$ is demi-normal by Example~\ref{exm: deminormal modifications}\eqref{exm: deminormal stein factorization}. By uniqueness of stable models together with~\cite[2.41 and Corollary 2.43(2)]{Kollar-singularities}, it is enough to prove the statement after taking the Galois double cover of $X$ as in Remark~\ref{rmk: self-intersection case} and~\cite[5.23]{Kollar-singularities}; in particular, we can assume the irreducible components of $X$ to be normal in codimension 1. By~\cite[Theorem 1.2]{Hashizume_sdlt_models}, there exists a crepant semi-dlt modification $\gamma\colon (X',\Delta_{X'})\rightarrow(X,\Delta_X)$ such that $\gamma$ does not contract any stratum of $D_{X'}$. Let $U=X\setminus\Ex(\varphi)$, $U'=\gamma^{-1}(U)$, and $V=Z\setminus\varphi(\Ex(\varphi))$. By connectedness of the fibers, we have $U=\varphi^{-1}(V)$. In particular, $(X',\Delta_{X'})|_V=(U',\Delta_{U'})$ admits a stable model over $V$, which is isomorphic to $V$ itself, and $U'\rightarrow V$ does not contract any stratum of the conductor of $U'$. Moreover, $R_{X'}\simeq R_X$, hence it is enough to show that $R_{X'}$ is finitely generated. Note that $R_{\Xbar'}\simeq R_{\Xbar}$ is a finitely generated $\cO_Z$-algebra, and the induced birational map $\Xbar\dashrightarrow\Proj R_{\Xbar}$ is the MMP step associated to the contraction $\Xbar\rightarrow\Zbar$. Moreover, $(X'\setminus U')\cap D_{X'}$ contains no log center. Indeed, let $W\subset(X'\setminus U')\cap D_{X'}$ be a log center, then $\gamma(W)\subset(X\setminus U)\cap D_X=\Ex(\varphi)\cap D_X$ would also be a log center as $\gamma$ is crepant, which would contradict the assumption. Then, $R_{X'}\simeq R_X$ is finitely generated by Theorem~\ref{thm: general version}, and the stable model $X\dashrightarrow\Proj R_X$ over $Z$ is an MMP step. 
\end{proof}

\bibliographystyle{amsalpha}
\bibliography{library}

\end{document}